\documentclass[12pt]{amsart} 

\usepackage{amscd} 
\usepackage{amsmath} 
\usepackage{amssymb} 
\usepackage{amsthm}
\usepackage{bigdelim}
\usepackage{color} 
\usepackage{enumerate}
\usepackage{mathrsfs}
\usepackage{multirow}
\usepackage[all]{xy} 
\usepackage{hyperref}
\newtheorem{thm}{Theorem}[section] 
\newtheorem{cor}[thm]{Corollary}
\newtheorem{prop}[thm]{Proposition}

\newtheorem{lem}[thm]{Lemma}
\theoremstyle{definition} 
\newtheorem{defn}[thm]{Definition}

\newtheorem{eg}[thm]{Example} 
\theoremstyle{remark}
\newtheorem{rem}[thm]{Remark}

\newtheorem{ques}[thm]{Question}

\newtheorem{cl}{Claim}
\newtheorem{step}{Step}

\newtheorem*{ack}{Acknowledgements}

\title{Splitting of algebraic fiber spaces with nef relative anti-canonical divisor and \\ decomposition of $F$-split varieties}
\author{Sho Ejiri}
\address{Department of Mathematics, Graduate School of Science, Osaka Metropolitan University, Osaka City, Osaka 558-8585, Japan}
\email{shoejiri.math@gmail.com}

\begin{document}
\maketitle
\markboth{SHO EJIRI}{Algebraic fiber spaces with nef relative anti-canonical divisor in positive characteristic}
\begin{abstract}
In this paper, we prove that an algebraic fiber space $f:X\to Y$ over a perfect field $k$ of characteristic $p>0$ with nef relative anti-canonical divisor $-K_{X/Y}$ splits into the product after taking the base change along a finite cover of $Y$, if the geometric generic fiber has mild singularities and if one of the following conditions holds:~(i) $k\subseteq\overline{\mathbb F_p}$;~(i\hspace{-1pt}i) $\pi^{\textup{\'et}}(Y)$ is finite; (i\hspace{-1pt}i\hspace{-1pt}i) $-K_X$ is semi-ample and $Y$ is an abelian variety. 
As its application, we generalize Patakfalvi and Zdanowicz's Beauville--Bogomolov decomposition in positive characteristic to the case when the anti-canonical divisor is numerically equivalent to a semi-ample divisor, which is applied to study the abundance conjecture in a new case and the fundamental group of an $F$-split variety with semi-ample anti-canonical divisor. 
We also show that a variety over a finite field with nef anti-canonical divisor satisfying some conditions has a rational point. 
Furthermore, to show the splitting theorem, we generalize partially Popa and Schnell's global generation theorem and Viehweg's weak positivity theorem to the case of generalized pairs in positive characteristic. 
\end{abstract}
\tableofcontents
\section{Introduction} \label{section:intro}
\subsection{Positivity of relative anti-canonical divisors}
The positivity of the relative anti-canonical divisor $-K_{X/Y}=-K_X+f^*K_Y$ restricts the geometric structure of an algebraic fiber space $f:X\to Y$. 
Here, an algebraic fiber space is a surjective morphism between smooth projective varieties with connected fibers that induces a separable extension of function fields.
Koll\'ar, Miyaoka and Mori~\cite[Corollary~2.8]{KMM92} proved that if $-K_{X/Y}$ is ample, then $Y$ is a point.
The same statement was shown in \cite[Corollary~4.10]{Eji19w} in the case when $-K_{X/Y}$ is nef and big. 
In characteristic zero, such theorems are generalized in terms of the augmented base locus $\mathbb B_+(-K_{X/Y})$ in \cite[Theorem~1.1]{EIM23}. 
 
Over the field of complex numbers $\mathbb C$, an argument due to Cao and H\"oring~\cite{CH19} implies that if $-K_{X/Y}$ is nef, then $f$ is locally constant, and in particular it is locally trivial. (For the singular case, see \cite{CCM21} and \cite[Appendix~A]{PZ19}.)
This was generalized by using the diminished base locus $\mathbb B_-(-K_{X/Y})$ in \cite[Theorem~1.4]{EIM23}. 
Also, applying Ambro's results~\cite[Proposition~4.4 and~Theorem~4.7]{Amb05} to the case when $-K_{X/Y}$ is semi-ample, we see that $f$ splits into the product after taking the base change along a finite \'etale cover of $Y$ (see also \cite[Theorem~1.7]{EIM23}). 
 
The study of such algebraic fiber spaces in positive characteristic started from Patakfalvi and Zdanowicz's pioneering work~\cite{PZ19}. They proved that if $-K_{X/Y}$ is nef and the geometric generic fiber $X_{\overline\eta}$ is strongly $F$-regular, where strong $F$-regularity is one of singularities defined by using the Frobenius morphism (Definition~\ref{defn:SFR}) that is closely related to klt singularity, then $f$ is a flat morphism with reduced fibers \cite[\S 4]{PZ19}.  
Additionally, their argument shows further that if either $-K_{X/Y}$ is semi-ample or the base field $k$ is contained in $\overline{\mathbb F_p}$, then $f$ splits into the product after taking the base change along a surjective flat morphism $I\to Y$ from a quasi-projective scheme $I$ (the isom scheme) \cite[\S 9]{PZ19}. In particular, we obtain that all the fibers of $f$ over $k$-rational points are isomorphic. 
 
In this paper, we study an algebraic fiber space $f:X\to Y$ with nef relative anti-canonical divisor $-K_{X/Y}$ in positive characteristic, and give sufficient conditions for $f$ to split after taking the base change along a finite cover of $Y$. They include a strengthening of the above statement in the case when $k\subseteq\overline{\mathbb F_p}$ (Theorem~\ref{thm:decomp1-intro}). 
Although we only introduce in this section the statements in the smooth case, 
we also deal in this paper with the singular case and the log case. 
\begin{thm}[\textup{Special case of Theorem~\ref{thm:decomp1}}] \label{thm:decomp1-intro} 
Let the base field $k$ be a perfect field of characteristic $p>0$. 
Let $f:X\to Y$ be a surjective morphism between smooth projective varieties. 
Suppose that $-K_{X/Y}$ is nef and the geometric generic fiber of $f$ is strongly $F$-regular. 
Then the following hold:
\begin{enumerate}[$(1)$]
\item If $Y$ is separably rationally connected, then $X\cong F\times_k Y$ as $Y$-schemes. 
\item If $\pi^{\textup{\'et}}(Y)=0$, then there exists an integer $e\ge 0$ such that $X\times_Y Y^e \cong F\times_k Y^e$ as $Y^e$-schemes. 
Here, $Y^e$ denotes the source of the $e$-times iterated Frobenius morphism. 
\item If either $k\subseteq \overline{\mathbb F_p}$ or $\pi^{\textup{\'et}}(Y)$ is finite, then there exists a finite surjective morphism $Z\to Y$ that is the composite of a finite \'etale cover and an iterated Frobenius morphism such that $X\times_Y Z \cong F\times_k Z$ as $Z$-schemes. 
\end{enumerate}
Here, $F$ in the above statements is $($a disjoint union of copies of$)$ a strongly $F$-regular projective variety. 
\end{thm}
The assumption that the geometric generic fiber $X_{\overline\eta}$ is strongly $F$-regular is satisfied if, for example, $X$ is $F$-split (Definition~\ref{defn:F-split}). In this case, $X_{\overline\eta}$ is indeed normal and $F$-split, so $F$-pure, and thus we can apply \cite[Theroem~6.3]{PZ19}. 
Here, $F$-purity is one of singularities defined by using Frobenius morphism (Definition~\ref{defn:F-pure}) that is closely related to log canonical singularity. 
 
If either~(i) we allow $X_{\overline\eta}$ to have ``bad'' singularities or~(i\hspace{-1pt}i) $k\not\subseteq \overline{\mathbb F_p}$ and $\pi^{\textup{\'et}}(Y)$ is not finite, then there exists an algebraic fiber space with nef relative anti-canonical divisor that does NOT split after taking the base change along any proper surjective morphism (Examples~\ref{eg:not F_p},~\ref{eg:not SFR} and~\ref{eg:F-pure but not SFR}).  
As for case~(i), however, a positive result has been known when $Y$ is an elliptic curve \cite[Theorem~1.6]{EP23}. 
On the other hand, negative examples are known for the case when $Y$ is a curve of genus at least two (Example~\ref{eg:not SFR}). 
If we restrict ourselves to the case when $Y$ is a curve, then the remaining is the case of $Y\cong\mathbb P^1$. 
In this case, we prove the following positive theorem without assuming the strong $F$-regularity of $X_{\overline\eta}$:
\begin{thm}[\textup{Special case of Theorem~\ref{thm:decomp P^1}}] \label{thm:decomp P^1-intro} 
Let $k$ and $f:X\to Y$ be as in Theorem~\ref{thm:decomp1-intro}. 
Assume that $Y\cong\mathbb P^1$.
Suppose that $-K_{X/Y}$ is nef. 
Then $X\cong X_y\times_k Y$ as $Y$-schemes, where $X_y$ is the fiber of $f$ 
over a $k$-rational point $y\in Y(k)$. 
\end{thm}
 
Let us return to the case of higher dimension. 
When we take Albanese morphisms into account, it is important to study the case when $Y$ is an abelian variety. In this case, we show that $f$ has isomorphic fibers: 
\begin{thm}[\textup{Special case of Theorem~\ref{thm:decomp2}}] \label{thm:decomp2-intro} 
Let $k$ and $f:X\to Y$ be as in Theorem~\ref{thm:decomp1-intro}. 
Assume that $Y$ is an abelian variety. 
Suppose that $-K_X$ is nef and $X_{\overline\eta}$ is strongly $F$-regular. 
Then all the fibers of $f$ over $k$-rational points are isomorphic. 
\end{thm}
Patakfalvi and Zdanowicz also proved that if $K_X$ is numerically trivial, if $X_{\overline\eta}$ is strongly $F$-regular and if $Y$ is an abelian variety, then $f$ splits into the product after taking the base change along an isogeny \cite[Proof of Theorem~11.6]{PZ19}. This was used to show the Beauville--Bogomolov decomposition in positive characteristic \cite[Theorem~1.1]{PZ19}. 
We show the same statement holds when $-K_X$ is semi-ample:
\begin{thm}[\textup{Special case of Theorem~\ref{thm:decomp3}}] \label{thm:decomp3-intro} 
Let $k$ and $f:X\to Y$ be as in Theorem~\ref{thm:decomp1-intro}. 
Assume that $Y$ is an abelian variety. 
Suppose that $-K_X$ is semi-ample and 
$X_{\overline\eta}$ is strongly $F$-regular. 
Then there exists an isogeny $Z\to Y$ such that 
$
X\times_Y Z \cong X_0\times_k Z
$
as $Z$-schemes, where $X_0$ is the fiber of $f$ over the identity element $0\in Y(k)$. 
\end{thm}
The proof does not use the isom scheme, which differs from that of \cite[Theroem~11.6]{PZ19}. 
\subsection{Decomposition theorems}
Mimicking the argument in \cite[Proof of Theorem~11.6]{PZ19}, 
we obtain from Theorems~\ref{thm:decomp1-intro} and~\ref{thm:decomp3-intro} the following decomposition theorem, which generalizes \cite[Theorem~1.1]{PZ19} to the case when the anti-canonical divisor is numerically equivalent to a semi-ample $\mathbb Q$-divisor.
\begin{thm}[\textup{Special case of Theorems~\ref{thm:F-split semi-ample} and \ref{thm:F-split nef}}] \label{thm:F-split-intro} 
Let $X$ be a smooth $F$-split projective variety over a perfect field $k$ of characteristic $p>0$. 
Suppose that $-K_X$ is numerically equivalent to 
a semi-ample $\mathbb Q$-divisor on $X$ 
$($resp. nef and $k\subseteq \overline{\mathbb F_p}$$)$. 
Then there exist finite surjective morphisms
$$
Z\times_{k'}A\xrightarrow{\tau}Y\xrightarrow{\sigma}X
$$
with the following properties:
\begin{enumerate}[$(1)$]
\item $k'$ is a finite field extension of $k$. 
\item $\sigma:Y\to X$ is \'etale. 
\item $\tau:Z\times_{k'}A\to Y$ is a torsor under $\Pi_{i=1}^{\hat q(X)}\mu_{p^{j_i}}$ for some integers $j_i\ge0$. 
\item $K_{Z\times_{k'}A}\sim \tau^*\sigma^*K_X$. 
\item $A$ is an ordinary abelian variety over $k'$ of dimension $\hat q(X)$. 
\item $Z$ is a Gorenstein strongly $F$-regular $F$-split geometrically integral projective variety over $k'$ such that $-K_Z$ is semi-ample $($resp. nef$)$ and $\hat q(Z)=0$.
\end{enumerate}
\end{thm}
Here, $\hat q(X)$ denotes the \textit{augmented irregularity} of $X$ defined as 
$$
\hat q(X):=\max \{\dim \mathrm{Alb}_{X'}|\textup{$X'\to X$ is a finite \'etale morphism}\}.
$$
When $-K_X$ is nef, we see from \cite[Theorem~1.1]{EP23} that $\hat q(X) \le \dim X$.

Theorem~\ref{thm:F-split-intro} cannot be generalized to each of the following cases:~(i) $X$ is not $F$-split (\cite[Proposition~1.3]{PZ19}); (i\hspace{-1pt}i) $k\not\subseteq\overline{\mathbb F_p}$ and $-K_X$ is nef but not semi-ample (Example~\ref{eg:indecomp}). 
Also, ``the torsor part'' cannot be dropped from Theorem~\ref{thm:F-split-intro} (Example~\ref{eg:torsor need}). 
\begin{rem} \label{rem:Delta}
It is well known that the $F$-splitting of $X$ implies that there exists an effective $\mathbb Q$-divisor on $X$ such that $K_X+\Delta\sim_{\mathbb Q} 0$ and $(X,\Delta)$ is $F$-split. Hence, if the pair $(X,\Delta)$ is strongly $F$-regular, then we obtain from \cite[Theorem~1.14]{PZ19} the same conclusion as Theorem~\ref{thm:F-split-intro}. However, it is not known whether or not we can choose such $\Delta$ so that $(X,\Delta)$ is strongly $F$-regular, even when $-K_X$ is semi-ample. 
\end{rem}
Combining Theorem~\ref{thm:F-split-intro} with a recent study of varieties with nef tangent bundle \cite{KW23, EY23}, we obtain the following decomposition theorem:
\begin{thm}[\textup{Theorem~\ref{thm:nef tangent}}] 
\label{thm:nef tangent-intro}
Let $k$ be an algebraically closed field of characteristic $p>0$. 
Let $X$ be a smooth $F$-split projective variety over $k$ 
with nef tangent bundle $T_X$. 
Suppose that either $k=\overline{\mathbb F_p}$ or 
$-K_X$ is numerically equivalent to a semi-ample $\mathbb Q$-divisor on $X$. 
Then there exist finite surjective morphisms
$$
Z\times_k A \xrightarrow{\tau} Y \xrightarrow{\sigma} X
$$
with the following properties: 
\begin{enumerate}[$(1)$]
\item $\sigma:Y\to X$ is \'etale. 
\item $\tau:Z\times_k A\to Y$ is a torsor under 
$\Pi_{i=1}^{\hat q(X)}\mu_{p^{j_i}}$ for some integers $j_i\ge 0$. 
\item $K_{Z\times_k A}\sim\tau^*\sigma^*K_X$. 
\item $A$ is an ordinary abelian variety. 
\item $Z$ is a smooth $F$-split separably rationally connected Fano variety
with nef tangent bundle. 
\end{enumerate}
In particular, $-K_X$ is semi-ample. 
\end{thm}
We cannot drop from this theorem each of the following:~(i) the assumption that $-K_X$ is semi-ample or $k=\overline{\mathbb F_p}$ (Example~\ref{eg:indecomp});~(i\hspace{-1pt}i) ``the torsor part'' (Example~\ref{eg:torsor need}). 

We also treat the case when $X$ is not necessarily $F$-split, 
and obtain the decomposition theorem in the case of $\hat q(X)=1$. 
\begin{thm}[\textup{Special case of Theorem~\ref{thm:hat q=1}}] \label{thm:hat q=1-intro} 
Let $X$ be a smooth projective variety $($not necessarily $F$-split$)$ over a perfect field $k$ of characteristic $p>0$. 
Suppose that $\hat q(X)=1$ and $-K_X$ is numerically equivalent to 
a semi-ample $\mathbb Q$-divisor on $X$. 
Then there exist finite surjective morphisms
$$
Z\times_{k'}E\xrightarrow{\tau}Y\xrightarrow{\sigma}X
$$
with the following properties:
\begin{enumerate}[$(1)$]
\item $k'$ is a finite field extension of $k$. 
\item $\sigma:Y\to X$ is \'etale. 
\item $\tau:Z\times_{k'}E\to Y$ is an infinitesimal torsor. 
\item $K_{Z\times_{k'}E}\sim \tau^*\sigma^*K_X$. 
\item $E$ is an elliptic curve over $k'$. 
\item $Z$ is a Gorenstein geometrically integral projective variety over $k'$ such that $-K_Z$ is semi-ample and $\hat q(Z)=0$.
\end{enumerate}
\end{thm}
This theorem cannot be generalized to each of the following cases: 
(i) $\hat q(X)\ge 2$ (\cite[Proposition~1.3]{PZ19}); 
(i\hspace{-1pt}i) $-K_X$ is nef but not semi-ample (Example~\ref{eg:indecomp}).
Also, ``the torsor part'' cannot be dropped from Theorem~\ref{thm:hat q=1-intro} (Example~\ref{eg:torsor need}). 
\subsection{Applications}
The above decomposition theorems have several applications. 
First, we enumerate direct corollaries of Theorems~\ref{thm:F-split-intro} and~\ref{thm:hat q=1-intro}. 
\begin{cor}[\textup{Special case of Corollary~\ref{cor:F-split semi-ample}}] \label{cor:F-split semi-ample-intro}
Let $X$ be a smooth $F$-split projective variety 
over a perfect field of positive characteristic. 
Suppose that $-K_X$ is numerically equivalent to 
a semi-ample $\mathbb Q$-divisor on $X$. 
Then $-K_X$ is semi-ample. 
\end{cor} 
\begin{cor}[\textup{Special case of Corollary~\ref{cor:hat q=1}}] \label{cor:hat q=1-intro}
Let $X$ be a smooth projective variety $($not necessarily $F$-split$)$
over a perfect field of positive characteristic. 
Suppose that $\hat q(X)\le 1$ and $-K_X$ is numerically equivalent to 
a semi-ample $\mathbb Q$-divisor on $X$. 
Then $-K_X$ is semi-ample.
In particular, if $\hat q(X)\le 1$ and $K_X\equiv 0$, 
then $K_X\sim_{\mathbb Q}0$. 
\end{cor}
\begin{cor}[\textup{Special case of Corollary~\ref{cor:surface}}] \label{cor:surface-intro}
Let $X$ be a smooth projective surface $($not necessarily $F$-split$)$
over a perfect field of positive characteristic. 
Suppose that $-K_X$ is numerically equivalent to 
a semi-ample $\mathbb Q$-divisor on $X$. 
Then $-K_X$ is semi-ample. 
\end{cor}
The last statement of Corollary~\ref{cor:hat q=1-intro} says that the abundance conjecture holds true when $\hat q(X)\le1$ and $K_X\equiv 0$. 
 
Theorems~\ref{thm:F-split-intro} and~\ref{thm:nef tangent-intro} have applications to the study of \'etale fundamental groups, which is an imitation of \cite[Corollary~1.12]{PZ19}. 
\begin{cor}[\textup{Corollary~\ref{cor:pi}}] \label{cor:pi-intro}
Let $X$ be a smooth $F$-split projective variety over an algebraically closed field $k$ of characteristic $p>0$. 
Suppose that either $-K_X$ is semi-ample or 
$-K_X$ is nef and $k=\overline{\mathbb F_p}$. 
If $\hat q(X)\ge \dim X-2$ $($e.g., $\dim X=3$ and $\hat q(X)>0$$)$, 
then $\pi^{\textup{\'et}}(X)$ is virtually abelian. 
I.e., there exists an abelian subgroup of finite index. 
\end{cor}
\begin{cor}[\textup{Corollary~\ref{cor:pi nef tangent}}] 
\label{cor:pi nef tangent-intro}
Let $k$ be an algebraically closed field of characteristic $p>0$. 
Let $X$ be a smooth $F$-split projective variety over $k$ 
with nef tangent bundle $T_X$. 
Suppose that either $-K_X$ is semi-ample or $k=\overline{\mathbb F_p}$. 
Then $\pi^{\textup{\'et}}(X)$ is virtually abelian. 
\end{cor}
Theorem~\ref{thm:decomp1-intro} can be used to find a rational point over a finite field. The following theorem is a combination of Theorem~\ref{thm:decomp1-intro} and \cite[Theorem~1.3 and~Proposition~3.4]{BF23}. 
\begin{cor}[\textup{Special case of Corollary~\ref{cor:rational points}}]
	\label{cor:rational points-intro}
Let $X$ be a geometrically integral smooth $F$-split 
projective variety over $\mathbb F_q$ for some $q=p^e>19$. 
Suppose that either 
\begin{enumerate}[$(a)$]
\item $-K_X$ is nef, $K_X\not\sim_{\mathbb Q}0$, $\frac{1}{2}b_1(X)\ge\dim X-3$ $($e.g., $\dim X =4$ and $b_1(X)>0$$)$ and $p>5$, or  
\item $K_X \equiv 0$ and $\frac{1}{2}b_1(X)\ge \dim X-2$ $($e.g., $\dim X=3$ and $b_1(X)>0$$)$. 
\end{enumerate}
Then $X(\mathbb F_q)\ne \emptyset$. 
\end{cor}
Here, $b_1(X):=\dim_{\mathbb Q_l}H^1_{\textup{\'et}}(X,\mathbb Q_l)$ for a prime $l\ne p$, which is known to be equal to the twice of the dimension of the Albanese variety. 
\subsection{Questions}
Let us have a rest from introducing results and consider three questions:
\begin{ques} \label{ques:1}
Let $k$ and $f:X\to Y$ be as in Theorem~\ref{thm:decomp1-intro}. Suppose that $-K_{X/Y}$ is nef and $X_{\overline\eta}$ is strongly $F$-regular. Then are all the fibers of $f$ over $k$-rational points isomorphic? I.e., can we drop the assumption that $Y$ is an abelian variety from Theorem~\ref{thm:decomp2-intro}?
\end{ques}
An argument over $\mathbb C$ that derives the local constancy of $f$ from a characteristic zero counterpart of Theorem~\ref{thm:num flat-intro} below is known (cf. \cite[Proposition~2.1]{Wan22a}), which implies all the fibers are isomorphic. However, the argument uses universal covers, so it cannot be applied to the case of positive characteristic. 
\begin{ques} \label{ques:2}
Let $k$ and $f:X\to Y$ be as in Theorem~\ref{thm:decomp1-intro}. Suppose that $-K_{X/Y}$ is semi-ample and $X_{\overline\eta}$ is strongly $F$-regular. Then does there exist a \textit{finite} cover of $Y$ such that $f$ splits into the product after taking the base change along the cover?
\end{ques}
As mentioned above, the answer is affirmative in characteristic zero \cite{Amb05}. In positive characteristic, $f$ splits into the product after taking the base change along a surjective flat morphism $I\to Y$ from the isom scheme $I$ (\cite[\S 9]{PZ19}), but it is not known if $I\to Y$ is finite. 
\begin{ques} \label{ques:3}
Let $X$ be a smooth projective variety over an algebraically closed field of positive characteristic. Suppose that $-K_X$ is numerically equivalent to a semi-ample $\mathbb Q$-divisor. Then is $-K_X$ semi-ample?
\end{ques}
In characteristic zero, it follows from the Bertini theorem and Gongyo's abundance theorem \cite[Theorem~1.2]{Gon13} that, for a log canonical projective pair $(X,\Delta)$, if $-K_X-\Delta$ is numerically equivalent to a semi-ample $\mathbb Q$-Cartier divisor, then $-K_X-\Delta$ is semi-ample.  
The same argument does not work in positive characteristic, since the Bertini theorem does not hold and the abundance theorem in this case is not known. 
\subsection{Positivity theorems}
We return to the introduction of theorems. 
Theorems~\ref{thm:decomp1-intro}, \ref{thm:decomp2-intro} and \ref{thm:decomp3-intro} are obtained as applications of the following theorem: 
\begin{thm}[\textup{Special case of Theorem~\ref{thm:num flat}}] 
\label{thm:num flat-intro} 
Let $k$ and $f:X\to Y$ be as in Theorem~\ref{thm:decomp1-intro}. 
Suppose that $-K_{X/Y}$ is nef and $X_{\overline\eta}$ is strongly $F$-regular. Then there exists an $f$-ample divisor $A$ on $X$ such that 
$f_*\mathcal O_X(mA)$ is numerically flat for each $m\in\mathbb Z_{\ge 0}$. 
\end{thm}
Theorem~\ref{thm:decomp1-intro} follows almost directly from Theorem~\ref{thm:num flat-intro}. 
For example, in case~(1) ($Y$ is separably rationally connected), every numerically flat vector bundle on $Y$ is isomorphic to a direct sum of structure sheaves, so there exists an isomorphism 
$$
\bigoplus_{m\ge0}f_*\mathcal O_X(mA)
\cong 
\mathcal O_Y \otimes_k \bigoplus_{m\ge0} H^0(X,\mathcal O_X(mA))
$$
of $\mathcal O_Y$-algebras, which means that $f$ splits into the product. 
A similar argument works for cases~(2) and~(3) of Theorem~\ref{thm:decomp1-intro}. 

Theorem~\ref{thm:num flat-intro} is obtained by using a partial generalization of Viehweg's weak positivity theorem \cite[Folgerung~3.4]{Vie82} to the case of generalized pairs in positive characteristic. This is established over an $F$-finite field, i.e., a field of characteristic $p>0$ with $[k:k^p]<\infty$. 
\begin{thm}[\textup{Theorem~\ref{thm:wp}}] \label{thm:wp-intro} 
Let the base field be an $F$-finite field of characteristic $p>0$. 
Let $f:X\to Y$ be a surjective morphism from a normal projective variety $X$ to a regular projective variety $Y$. 
Let $\Delta$ be an effective $\mathbb Q$-Weil divisor on $X$. 
Let $M$ be a nef $\mathbb Q$-Cartier divisor on $X$. 
Let $V$ be an open subset of $Y$ and set $U:=f^{-1}(V)$. 
Suppose the following conditions:  
\begin{itemize}
\item $i(K_X+\Delta)$ and $iM$ are Cartier for an integer $i>0$ not divisible by $p$; 
\item $K_X+\Delta+M$ is $f$-ample; 
\item $U$ is flat over $V$; 
\item $\mathrm{Supp}(\Delta)$ does not contain any component of any fiber over $V$;
\item every fiber over $V$ is geometrically normal;
\item $\left(X_{\overline y}, \Delta|_{X_{\overline y}}\right)$ is $F$-pure for every $y\in V$, where $X_{\overline y}$ is the geometric fiber of $f$over $y$. 
\end{itemize}
Then there exists an $m_0\in\mathbb Z_{>0}$ such that 
$$
f_*\mathcal O_X(m(K_{X/Y}+\Delta+M)+N)
$$
is weakly positive over $V$ for each $m\ge m_0$ with $i|m$ and every nef Cartier divisor $N$ on $X$. 
\end{thm}
For the definition of weak positivity, see Definition~\ref{defn:wp}.

The reason why Theorem~\ref{thm:wp-intro} is a ``partial'' generalization is that the moduli part $M$ of the generalized pair $(X, \Delta+M)$ is a nef $\mathbb Q$-Cartier divisor, while we do not include the case where $M$ is only a push-forward of a nef $\mathbb Q$-Cartier divisor on a higher birational model of $X$. 

To prove Theorem~\ref{thm:wp-intro}, we partially generalize Popa and Schnell's global generation theorem (\cite[Variant~1.6]{PS14}) to the case of generalized pairs. 
\begin{thm}[\textup{Theorem~\ref{thm:PS-type}}] \label{thm:PS-type-intro} 
Let the base field be an $F$-finite field of characteristic $p>0$. 
Let $(X,\Delta)$ be an $F$-pure projective pair. 
Let $f:X\to Y$ be a morphism to a projective variety $Y$. 
Let $M$ be a nef $\mathbb Q$-Cartier divisor on $X$. 
Suppose that 
\begin{itemize}
\item $i(K_X+\Delta)$ and $iM$ are Cartier for an integer $i>0$ not divisible by $p$, and  
\item $K_X+\Delta+M$ is $f$-ample. 
\end{itemize} 
Let $L$ be an ample Cartier divisor on $Y$, and let $j$ be the smallest positive integer such that $|jL|$ is free. 
Then there exists an $m_0\in \mathbb Z_{>0}$ such that 
$$
f_*\mathcal O_X(m(K_X+\Delta+M)+N) \otimes \mathcal O_Y(lL)
$$
is generated by its global sections for each $m\ge m_0$ with $i|m$, each $l\ge m(jn+1)$ and every nef Cartier divisor $N$ on $X$. 
\end{thm}
This paper is organized as follows. 
In Section~\ref{section:notation}, notation and terminology are explained. 
In Section~\ref{section:positivity theorems}, we consider a property of the trace maps of iterated Frobenius morphisms and prove Theorems~\ref{thm:wp-intro} and~\ref{thm:PS-type-intro}. 
Section~\ref{section:nef relative anti} is devoted to show Theorems~\ref{thm:decomp1-intro} and \ref{thm:num flat-intro}. 
In Section~\ref{section:over abelian}, we deal with algebraic fiber spaces over abelian varieties and prove Theorems~\ref{thm:decomp2-intro} and \ref{thm:decomp3-intro}. 
In Section~\ref{section:over elliptic}, we treat algebraic fiber spaces over elliptic curves whose geometric generic fibers may have bad singularities. 
In Section~\ref{section:decomposition}, we establish a decomposition theorem for $F$-split varieties whose anti-canonical divisors satisfy some positivity conditions, and apply it to study the abundance conjecture and \'etale fundamental groups. 
In Section~\ref{section:P^1}, we consider an algebraic fiber space over $\mathbb P^1$ without assuming that the geometric generic fiber is strongly $F$-regular, and show Theorem~\ref{thm:decomp P^1-intro}.
Section~\ref{section:examples} is used to collect examples which ensures that the assumptions of Theorems~\ref{thm:decomp1-intro}, \ref{thm:F-split-intro}, \ref{thm:nef tangent-intro} and \ref{thm:hat q=1-intro} cannot be dropped. 
\begin{ack}
The author is greatly indebted to Shou Yoshikawa for answering several questions. 
He wishes to express his thanks to Fabio Bernasconi for reading the preprint version of this paper and pointing out that the assumption of smoothness in Corollary~\ref{cor:rational points} can be dropped. 
He would like to thank Osamu Fujino, Masataka Iwai, Shin-ichi Matsumura, Zsolt Patakfalvi and Lei Zhang for valuable comments. 
He also would like to thank Fabio Bernasconi and Stefano Filipazzi for sharing \cite{BF23}. 
He was partly supported by MEXT Promotion of Distinctive Joint Research Center Program JPMXP0723833165.
\end{ack}
\section{Notation and terminology} \label{section:notation}
Let $k$ be a field. %
By \textit{$k$-scheme} we mean a separated scheme of finite type over $k$. 
A \textit{variety} is an integral $k$-scheme.
Let $X$ be an equi-dimensional $k$-scheme satisfying $S_2$ and $G_1$. 
Here, $S_2$ (resp. $G_1$) is Serre's second condition 
(resp. the condition that it is Gorenstein in codimension $1$). 
Let $\mathcal K$ be the sheaf of total quotient rings on $X$. 
An \textit{AC divisor} (or almost Cartier divisor) on $X$ is 
a reflexive coherent subsheaf of $\mathcal K$ which is invertible on 
an open subset of $X$ whose complement has codimension at least two 
(\cite[p. 301]{Har94}, \cite[Definition 2.1]{MS12}). 

Let $D$ be an AC divisor on $X$. 
We let $\mathcal O_X(D)$ denote the coherent sheaf defining $D$. 
We say that $D$ is \textit{effective} if $\mathcal O_X \subseteq O_X (D)$. 
The set $\mathrm{WSh}(X)$ of AC divisors on $X$ forms naturally 
an additive group (\cite[Corollary 2.6]{Har94}). 
In this paper, a \textit{prime} AC divisor is an effective AC divisor 
that cannot be written as the sum of two non-zero effective AC divisors. 

A \textit{$\mathbb Q$-AC divisor} is an element of $\mathrm{WSh}(X)\otimes_{\mathbb Z} \mathbb Q$. 
Let $\Delta$ be a $\mathbb Q$-AC divisor. 
Then there are prime AC divisors $\Delta_1,\ldots, \Delta_n$ on $X$ 
such that $\Delta= \sum_i \delta_i\Delta_i$. 
We define
$$
\lfloor \Delta \rfloor := \sum_i \lfloor\delta_i\rfloor \Delta_i, 
\quad \textup{and} \quad 
\lceil \Delta \rceil := \sum_i \lceil\delta_i\rceil \Delta_i. 
$$
Note that $\lfloor\Delta\rfloor$ and $\lceil\Delta\rceil$ are not necessarily uniquely determined by $\Delta$, 
because the choice of the decomposition $\Delta=\sum_i\delta_i\Delta_i$ 
is not necessarily unique. 
In this paper, for a $\mathbb Q$-AC divisor $\Delta$, 
we fix a decomposition $\Delta=\sum_{i=1}^n\delta_i\Delta_i$. 
If $\Delta$ is $\mathbb Q$-Cartier, 
we also fix a decomposition into Cartier divisors. 
If $\Delta=a\Delta'+b\Delta''$ for some $a,b \in \mathbb Q$ and 
$\mathbb Q$-AC divisors $\Delta'$ and $\Delta''$ whose decompositions 
$\Delta'=\sum_i \delta'_i \Delta'_i$ and 
$\Delta=\sum_i \delta''_i \Delta''_i$ have already been given, 
then we choose the natural decomposition 
$\Delta=\sum_i a\delta'_i\Delta'+b\delta''_i\Delta''_i$. 
We say that $\Delta$ is \textit{effective} if $\delta_i \ge 0$ for each $i$. 
By $\Delta\ge\Delta'$, we mean $\Delta-\Delta'$ is effective. 
We say that a $\mathbb Q$-AC divisor $\Delta=\sum_i\delta_i\Delta_i$ 
is \textit{integral} if $\delta_i \in\mathbb Z$ for each $i$. 
Replacing $\mathbb Q$ by $\mathbb Z_{(p)}$, we can define a similar notions. 

Let $\varphi: S \to T$ be a morphism of schemes 
and let $T'$ be a $T$-scheme. 
Let $S_{T'}$ denote the fiber product $S\times_T T'$ and 
let $\phi_{T'}$ be its second projection. 
For an $\mathcal O_S$-module $\mathcal G$, its pullback to $S_{T'}$ 
is denoted by $\mathcal G_{T'}$. 
We use the same notation for an AC or $\mathbb Q$-AC divisor 
if its pullback is well-defined.
 
Let $k$ be a field of positive characteristic. 
Let $X$ be a $k$-scheme. 
We denote by $F_X:X^1\to X$ the absolute Frobenius morphism of $X$. 
When we regard $X$ as the source of the $e$-times iterated Frobenius morphism $F_X^e$, we denote it by $X^e$. 
Let $f:X\to Y$ be a morphism of $k$-schemes. 
We denote the $e$-th relative Frobenius morphism by 
$F_{X/Y}^{(e)}:X^e\to X_{Y^e}$, which is given by 
$F_{X/Y}^{(e)}=\left(F_X^e, f^{(e)}\right):X^e\to X\times_Y Y^e$, 
where $f^{(e)}$ denote $f:X\to Y$ when we regard $X$ and $Y$ as $X^e$ and $Y^e$, respectively. 
\begin{defn}[\textup{$F$-purity, \cite[Definition~2.1]{HW02}}] \label{defn:F-pure}
Let $k$ be an $F$-finite field. 
Let $X$ be an equi-dimensional $k$-scheme satisfying $S_2$ and $G_1$. %
Let $\Delta$ be an effective $\mathbb Q$-AC divisor on $X$. %
We say that the pair $(X,\Delta)$ is \textit{$F$-pure} if for each $e\in\mathbb Z_{>0}$, the composite
\begin{align*} \tag{$\ast$}
\mathcal O_X 
\xrightarrow{{F_X^e}^\sharp}
{F_X^e}_*\mathcal O_X
\hookrightarrow 
{F_X^e}_*\mathcal O_X(\lfloor (p^e-1)\Delta \rfloor)
\end{align*}
locally splits as an $\mathcal O_X$-module homomorphism.
\end{defn}
\begin{defn}[\textup{strong $F$-regularity, \cite[Definition~3.1]{SS10}}] \label{defn:SFR}
Let $k$ be an $F$-finite field. 
Let $(X,\Delta)$ be an affine normal pair. 
We say that the pair $(X,\Delta)$ is \textit{strongly $F$-regular} if for every effective divisor $D$ on $X$, %
there exists an $e\in\mathbb Z_{>0}$ such that the composite %
$$
\mathcal O_X 
\xrightarrow{{F_X^e}^\sharp}
{F_X^e}_*\mathcal O_X
\hookrightarrow
{F_X^e}_*\mathcal O_X(\lceil(p^e-1)\Delta\rceil +D)
$$
splits as an $\mathcal O_X$-module homomorphism. 
Let $(X,\Delta)$ be a normal pair. 
We say that $(X,\Delta)$ is \textit{strongly $F$-regular} if there exists an affine open cover $\{V_i\}$ of $X$ such that $\left(V_i,\Delta|_{V_i}\right)$ is strongly $F$-regular for each $i$. 
\end{defn}
Let $k$ be an $F$-finite field. 
Let $X$ be an equi-dimensional $k$-scheme satisfying $S_2$ and $G_1$. 
Let $\Delta$ be an effective $\mathbb Z_{(p)}$-AC divisor on $X$. 
Take an $e\in\mathbb Z_{>0}$ with $(p^e-1)\Delta$ integral. 
Applying $\mathcal Hom(?, \mathcal O_X)$ to $(\ast)$, we obtain the 
morphism 
$$
\phi_{(X,\Delta)}^{(e)}:
{F_X^e}_*\mathcal O_X((1-p^e)(K_X+\Delta))
\to 
\mathcal O_X
$$
by the Grothendieck duality. 
Let $M$ be an AC divisor on $X$. 
Let $\phi_{(X,\Delta)}^{(e)}(M)$ denote the morphism obtained by taking 
the reflexive hull of the tensor product of 
$\phi_{(X,\Delta)}^{(e)}$ and $\mathcal O_X(M)$:
$$
\phi_{(X,\Delta)}^{(e)}(M):
{F_X^e}_*\mathcal O_X((1-p^e)(K_X+\Delta)+p^eM)
\to 
\mathcal O_X(M). 
$$
One can easily check that if $\phi_{(X,\Delta)}^{(e)}$ is surjective, 
then so is $\phi_{(X,\Delta)}^{(e)}(M)$. 

Let $f:X\to Y$ be a surjective morphism to a regular variety $Y$. 
By a procedure similar to the above, we can define $\phi_{(X/Y,\Delta)}^{(e)}(M_{Y^e})$ that is a relative version of $\phi_{(X,\Delta)}^{(e)}(M)$: 
$$
\phi_{(X/Y,\Delta)}^{(e)}(M_{Y^e}):
{F_{X/Y}^{(e)}}_* \mathcal O_X((1-p^e)(K_{X/Y}+\Delta) +p^e M) 
\to 
\mathcal O_{X_{Y^e}}(M_{Y^e}). 
$$
For a concrete explanation, see \cite[\S 3]{Eji19d}. 
\begin{defn}[\textup{Glogal generation over $V$}]
Let $Y$ be a variety over a field
and let $\mathcal G$ be a coherent sheaf on $Y$. 
Let $V$ be a subset of the underlying topological space $\mathrm{tp}(Y)$ of $Y$. 
We say that $\mathcal G$ is \textit{generated by its global sections over $V$}
(or \textit{globally generated over $V$}) 
if $V$ and the support of the cokernel of the natural map 
$$
H^0(Y,\mathcal G) \otimes \mathcal O_Y \to \mathcal G
$$
do not intersect. 
\end{defn}
\begin{defn}[\textup{Weak positivity, \cite[Definition~1.2]{Vie82}}]
	\label{defn:wp}
Let $Y$ be a normal quasi-projective variety over a field
and let $\mathcal G$ be a torsion-free coherent sheaf on $Y$. 
Let $V$ be an open subset of $Y$ such that $\mathcal G|_V$ is locally free. 
We say that $\mathcal G$ is \textit{weakly positive over $V$} if 
for every $\alpha\in\mathbb Z_{>0}$ and every ample Cartier divisor $H$ 
on $Y$, there exists a $\beta\in\mathbb Z_{>0}$ such that 
$(S^{\alpha\beta}(\mathcal G) )^{\vee\vee}(\beta H)$ 
is generated by its global sections over $V$. 
Here, $S^{\alpha\beta}(\mathcal G)$ denotes the $\alpha\beta$-th 
symmetric product of $\mathcal G$ and $(?)^{\vee\vee}$ denotes the double dual. 
\end{defn}
\begin{defn}[\textup{Numerical flatness}]
Let $\mathcal E<\delta>$ be a $\mathbb Q$-twisted vector bundle 
on a projective variety
(for the definition of $\mathbb Q$-twisted vector bundles, see \cite[Chapter 6, \S 2]{Laz04II}). 
We say that $\mathcal E<\delta>$ is \textit{numerically flat} if 
both $\mathcal E<\delta>$ and $\mathcal E^\vee<-\delta>$ are nef. 
\end{defn}
One can easily check that $\mathcal E<\delta>$ is numerically flat if 
$\mathcal E<\delta>$ is nef and $\det\mathcal E<r\delta>$ is 
numerically trivial, where $r:=\mathrm{rank}\,\mathcal E$. 
\section{Positivity theorems} \label{section:positivity theorems}
In this section, we work over an $F$-finite field $k$ of characteristic $p>0$ (i.e., a field of characteristic $p>0$ with $[k:k^p]<\infty$). 
\subsection{Global generation theorem of Popa--Schnell type}
In this subsection, we prove a positivity theorem of Popa--Schnell type. 
\begin{prop}[\textup{\cite[Corollary~2.23]{Pat14}}] \label{prop:surj}
Let $X$ be an equi-dimensional $k$-scheme satisfying $S_2$ and $G_1$. 
Let $\Delta$ be an effective $\mathbb Z_{(p)}$-AC divisor on $X$. 
Let $Y$ be a variety and let $f:X\to Y$ be a projective morphism. 
Let $L$ be an $f$-ample Cartier divisor on $X$. 
Suppose that 
\begin{itemize}
\item $K_X+\Delta$ is $\mathbb Q$-Cartier, and 
\item $(X,\Delta)$ is $F$-pure. 
\end{itemize}
Then there exists an $m_0\in\mathbb Z_{>0}$ such that 
$$
f_*\phi_{(X,\Delta)}^{(e)}(mL+N): 
f_*{F_X^e}_*\mathcal O_X((1-p^e)(K_X+\Delta)+p^e(mL+N))
\to f_*\mathcal O_X(mL+N)
$$
is surjective for each $m\ge m_0$, each $e\in\mathbb Z_{>0}$ with $(p^e-1)\Delta$ integral 
and every $f$-nef Cartier divisor $N$. 
\end{prop}
\begin{proof}
For simplicity, we denote $K_X+\Delta$ by $K$. 
Let $g$ be the smallest integer such that $(p^g-1)K$ is integral. 
For each $e\in\mathbb Z_{>0}$, the morphism $\phi_{(X,\Delta)}^{(ge)}$ 
can be decomposed into the surjective morphisms of the following form:
\begin{align*}
& {F_X^{g(j-1)}}_*\big(\phi_{(X,\Delta)}^{(g)}((1-p^{g(j-1)})K) \big):
\\ & {F_X^{gj}}_*\mathcal O_X((1-p^{gj})K )
\to
{F_X^{g(j-1)}}_*\mathcal O_X((1-p^{g(j-1)})K). 
\end{align*}
Here, $1\le j\le e$. 
Thus, it is enough to show that 
\begin{align*}
& f_*\left( {F_X^{g(j-1)}}_*
\big(\phi_{(X,\Delta)}^{(g)}((1-p^{g(j-1)})K) \big)
\otimes \mathcal O_X(mL+N) \right)
\\ \cong & f_*{F_X^{g(j-1)}}_*
\big(\phi_{(X,\Delta)}^{(g)}((1-p^{g(j-1)})K +p^{g(j-1)}(mL+N)) \big) 
\\ \cong & {F_Y^{g-1}}_*f_*
\big(\phi_{(X,\Delta)}^{(g)}((1-p^{g(j-1)})K +p^{g(j-1)}(mL+N)) \big)
\end{align*}
is surjective, which is equivalent to the surjectivity of 
$$
f_*\big(\phi_{(X,\Delta)}^{(g)}((1-p^{g(j-1)})K+p^{g(j-1)}(mL+N)) \big), 
$$
since $F_Y$ is affine. 
Let $i_C$ (resp. $i_W$) be the Cartier (resp. Weil) index of $K$. 
Set 
$$
\mathcal G:=\bigoplus_{0\le r<i_C,~i_W|r}
\mathrm{Ker}\left(\phi_{(X,\Delta)}^{(g)}(-rK) \right). 
$$
Note that $\phi_{(X,\Delta)}^{(g)}(-rK)$ is surjective for each $r$. 
By Keeler's relative Fujita vanishing theorem \cite[Theorem~1.5]{Kee03}, 
we can find an $m_1\in\mathbb Z_{>0}$ such that 
$$
R^1f_*\mathcal G(mL+N)=0
$$
for each $m\ge m_1$ and every $f$-nef Cartier divisor $N$ on $X$. 
Then 
$$
f_*\phi_{(X,\Delta)}^{(g)}(-rK+mL+N)
$$
is surjective for each $0\le r<i_C$ with $i_W|r$, 
each $m\ge m_1$ and every $f$-nef Cartier divisor $N$ on $X$. 
Take an $m_2 \in \mathbb Z_{>0}$ so that $m_2L-K$ is $f$-nef. 
Put $m_0:=m_1+m_2$. 
Let $q_j$ and $r_j$ be the quotient and the remainder of the division 
of $p^{g(j-1)}-1$ by $i_C$. 
Since $i_W|i_C$, we have $i_W|r_j$. 
Then, for each $m\ge m_0$, we get 
\begin{align*}
&f_*\big(\phi_{(X,\Delta)}^{(g)}((1-p^{g(j-1)})K+p^{g(j-1)}(mL+N))\big)
\\ \cong & 
f_*\big(\phi_{(X,\Delta)}^{(g)}(
(-r_j-q_ji_C)K+p^{g(j-1)}(mL+N)
)\big)
\\ \cong & 
f_*\big(\phi_{(X,\Delta)}^{(g)}(
-r_jK+\underbrace{q_ji_C(m_2L-K)}_{\textup{$f$-nef}}
+{(p^{g(j-1)}m-q_ji_Cm_2)}L+\underbrace{p^{g(j-1)}N}_{\textup{$f$-nef}}
)\big),  
\end{align*}
so we only need to check that $p^{g(j-1)}m-q_ji_Cm_2 \ge m_1$. 
If $q_j=0$, then it follows from $m\ge m_1$. 
If $q_j>0$, then 
$$
p^{g(j-1)}m-q_ji_Cm_2 > q_ji_Cm-q_ji_Cm_2=q_ji_Cm_1\ge m_1.
$$ 
\end{proof}
\begin{lem}[\textup{\cite[Lemma~3.4]{Eji19p}}] \label{lem:gg}
Let $f:X\to Y$ be a morphism between projective varieties. 
Let $A$ be an ample Cartier divisor on $X$. 
Let $\mathcal F$ be a coherent sheaf on $X$. 
Then there exists an $m_0\in\mathbb Z_{>0}$ such that 
$f_*\mathcal F(mA+N)$ is generated by its global sections 
for each $m\ge m_0$ and every nef Cartier divisor on $X$. 
\end{lem}
\begin{lem}[\textup{\cite[Lemma~3.4]{Eji22a}}] \label{lem:Fgg}
Let $Y$ be a projective variety of dimension $n$ and 
let $\mathcal G$ be a coherent sheaf on $Y$. 
Let $H$ be an ample Cartier divisor on $Y$ 
and let $j$ be a positive integer such that $|jH|$ is free. 
Let $\{a_e\}_{e\ge0}$ be a sequence of positive integers
such that $a_e/p^e$ converses to $\varepsilon +jn$ 
for some $\epsilon\in\mathbb R_{>0}$. 
Then there exists an $e_0\in\mathbb Z_{>0}$ such that 
${F_Y^e}_*(\mathcal G(a_eH))$ is generated by its global sections 
for each $e\ge e_0$. 
\end{lem}
\begin{thm} \label{thm:PS-type}
Let $X$ be an equi-dimensional projective $k$-scheme satisfying $S_2$ and $G_1$. 
Let $\Delta$ be an effective $\mathbb Q$-AC divisor. 
Let $Y$ be a projective variety of dimension $n$ 
and let $f:X\to Y$ be a morphism. 
Let $M$ be a nef $\mathbb Q$-Cartier divisor on $X$. 
Let $V$ be an open subset of $Y$. 
Set $U:=f^{-1}(V)$ and $f|_U:U\to V$.  
Suppose that 
\begin{itemize}
\item $\Delta|_U$ is a $\mathbb Z_{(p)}$-AC divisor,
\item $(U,\Delta|_U)$ is $F$-pure, 
\item $K_X+\Delta$ is $\mathbb Q$-Cartier, and  
\item $K_X+\Delta+M$ is $f$-ample. 
\end{itemize}
Let $H$ be an ample Cartier divisor on $Y$ and 
let $j$ be the smallest positive integer such that $|jH|$ is free. 
Then there exists an $m_0\in\mathbb Z_{>0}$ such that 
$$
f_*\mathcal O_X(m(K_X+\Delta+M)+N) \otimes \mathcal O_Y(lH)
$$
is generated by its global sections over $V$ 
for each $m\ge m_0$ such that $m(K_X+\Delta)$ and $mM$ are Cartier, 
each $l\ge m(jn+1)$ and 
every nef Cartier divisor $N$ on $X$. 
\end{thm}
\begin{proof}
Put $L:=K_X+\Delta+M$. 
Let $i$ be the smallest positive integer such that $i(K_X+\Delta)$ and $iM$ are Cartier. 
By Proposition~\ref{prop:surj}, there is an $m_0$ such that 
\begin{align*} \tag{$\ast$}
& f_*\phi_{(X,\Delta)}^{(e)}(mL+N): 
\\ & {F_Y^e}_*f_*\mathcal O_X(\lceil(1-p^e)(K_X+\Delta)\rceil+p^e(mL+N))
\to f_*\mathcal O_X(mL+N)
\end{align*}
is surjective over $V$ for each $m\ge m_0$ with $i|m$, 
each $e\in\mathbb Z_{>0}$ with $(p^e-1)\Delta|_U$ integral 
and every nef Cartier divisor $N$ on $X$. 
Let $q_e$ and $r_e$ be the quotient and the remainder of the division of 
$p^e-1$ by $i$. 
Then 
\begin{align*}
&\lceil(1-p^e)(K_X+\Delta)\rceil+p^e(mL+N)
\\ = & 
-q_ei(K_X+\Delta)-\lfloor r_e(K_X+\Delta)\rfloor +p^e(mL+N)
\\ \sim & 
(p^em-q_ei)L +q_ei M +p^eN -\lfloor r_e(K_X+\Delta)\rfloor. 
\end{align*}
Put $\mathcal G:=\bigoplus_{0\le r <i} \mathcal O_X(-\lfloor r(K_X+\Delta)\rfloor)$. 
Take a $\mu\in\mathbb Z_{>0}$ so that $i|\mu$, $\mu\ge m_0$ 
and $\mu L$ is $f$-free. 
Applying Keeler's relative Fujita vanishing \cite[Theorem~1.5]{Kee03} 
to $\mathcal G$ and using the relative Castelnuovo--Mumford regularity 
\cite[Example~1.8.24]{Laz04I}, 
we obtain an $l_0\in\mathbb Z_{>0}$ such that 
\begin{align*}
& f_*\mathcal O_X(\mu L) \otimes f_*\mathcal O_X(lL +q_eiM+p^eN -\lfloor r_e(K_X+\Delta)\rfloor )
\\ & \to 
f_*\mathcal O_X((\mu+l)L +q_eiM+p^eN -\lfloor r_e(K_X+\Delta)\rfloor)
\end{align*}
is surjective for each $l\ge l_0$ and each $e$.  
Take a $\nu \in\mathbb Z_{>0}$ so that $L+\nu f^*H$ is ample. 
Applying Lemma~\ref{lem:gg} to $\mathcal G$ and $i(L+\nu f^*H)$, 
we find an $l_1\in\mathbb Z_{>0}$ with $i|l_1$ and $l_0\le l_1$ such that 
\begin{align*}
& f_*\mathcal O_X(lL+l\nu f^*H +q_eiM +p^eN -\lfloor r_e(K_X+\Delta) \rfloor)
\\ & \cong 
f_*\mathcal O_X(lL +q_eiM+p^e N -\lfloor r_e(K_X+\Delta) \rfloor) (l\nu H)
\end{align*}
is globally generated for each $l\ge l_1$ and each $e$. 
We have the surjective morphism 
$$
\bigoplus \mathcal O_Y(-l\nu H)
\twoheadrightarrow
f_*\mathcal O_X(lL +q_eiM +p^eN -\lfloor r_e(K_X+\Delta) \rfloor)
$$
Let $s_e$ and $t_e$ be integers satisfying 
$p^em-q_ei =s_e\mu +t_e$ and $l_1\le t_e < l_1+\mu$. 
Note that $i|t_e$. 
Let $u$ be the smallest integer such that $f_*\mathcal O_X(\mu L)(uH)$ 
is globally generated over $V$. 
Then we have the morphism  
$ 
\bigoplus \mathcal O_Y(-uH)
\to
f_*\mathcal O_X(\mu L) 
$
that is surjective over $V$. 
Using them, we obtain the following sequence of morphisms 
that are surjective over $V$:
\begin{align*}
& \bigoplus \mathcal O_Y(-(s_eu+\nu t_e)H)
\\ \cong & \left( \bigoplus \mathcal O_Y(-s_eu H) \right)
\otimes \left( \bigoplus \mathcal O_Y(-\nu t_eH) \right)
\\ \to & 
\left( \bigotimes^{s_e} f_*\mathcal O_X(\mu L) \right)
\otimes f_*\mathcal O_X(t_eL+q_eiM+p^eN -\lfloor r_e(K_X+\Delta)\rfloor)
\\ \twoheadrightarrow & 
f_* \mathcal O_X((p^em-q_ei)L +q_eiM+p^eN -\lfloor r_e(K_X+\Delta)\rfloor)
\\ \cong & 
f_* \mathcal O_X(\lceil (1-p^e)(K_X+\Delta)\rceil +p^e(mL+N)). 
\end{align*}
Pushing forward this by $F_Y^e$ and combining with $(\ast)$, 
we obtain the morphism 
\begin{align*} 
\bigoplus {F_Y^e}_* \mathcal O_Y(-(s_eu+\nu t_e)H)
\to  f_*\mathcal O_X(mL+N)
\end{align*}
that is surjective over $V$. 
Taking tensor product with $\mathcal O_Y(lH)$ for an $l\in\mathbb Z$, 
we get 
\begin{align*} \tag{$\ast\ast$}
\bigoplus {F_Y^e}_* \mathcal O_Y((p^el-s_eu-\nu t_e)H)
\to  f_*\mathcal O_X(mL+N)(lH). 
\end{align*}
Since 
$$
\lim_{e\to\infty} \frac{p^el-s_eu-\nu t_e}{p^e} =l-\frac{m-1}{\mu}u, 
$$
if $l>\frac{m-1}{\mu}u +jn$, then 
$
{F_Y^e}_*\mathcal O_Y((p^el-s_eu -\nu t_e)H)
$
is globally generated by Lemma~\ref{lem:Fgg}, so 
$f_*\mathcal O_X(mL+N)(lH)$ is globally generated over $V$ by $(\ast\ast)$. 
We show that 
$$
m(jn+1) >\frac{m-1}{\mu}u+jn, 
$$ 
which implies our assertion.
When $m=\mu$ and $N=0$, if $l>\frac{m-1}{\mu}u+jn$, 
then $f_*\mathcal O_X(\mu L)(lH)$ is globally generated over $V$, 
so $u\le l$, which means that 
$$
u \le \frac{\mu-1}{\mu}u +jn +1, 
$$
and hence $u \le \mu(jn+1)$. 
Therefore, we get 
$$
\frac{m-1}{\mu}u +jn \le (m-1)(jn+1) +jn =m(jn+1) -1 < m(jn+1). 
$$
\end{proof}
\subsection{Weak positivity theorem}
To prove the weak positivity theorem, we introduce an invariant of coherent sheaves that measures positivity, 
which is a version of the one introduced in \cite{Eji17} adjusted to our situation. 
\begin{defn}[\textup{\cite[Definition~4.4]{Eji17}}] \label{defn:t}
Let $Y$ be a quasi-projective variety
and let $\mathcal G$ be a coherent sheaf on $Y$. 
Let $H$ be an ample Cartier divisor on $Y$ 
and let $V$ be a subset of the underlying topological space $\mathrm{tp}(Y)$ of $Y$. 
We define 
\begin{align*}
T_V(\mathcal G, H) & :=
\left\{\varepsilon\in\mathbb Z\left\lbrack\frac{1}{p}\right\rbrack \middle|
\begin{tabular}{c}
\textup{$\exists e\in\mathbb Z_{>0}$ s.t. $p^e\varepsilon \in\mathbb Z$ and
$\left({F_Y^e}^*\mathcal G\right)(-p^e\varepsilon H) $ } \\
\textup{is globally generated over $V$}
\end{tabular}
\right\}, 
\\ t_V(\mathcal G,H) &:= \sup T_V(\mathcal G,H), \textup{~and} 
\\ t(\mathcal G,H) &:= t_{\{\eta\}}(\mathcal G,H). 
\end{align*}
Here, $\eta$ is the generic point of $Y$. 
\end{defn}
\begin{lem} \label{lem:t-basic}
Let $Y$, $\mathcal G$, $V$ and $H$ be as in Definition~\ref{defn:t}. 
\begin{enumerate}[$(1)$]
\item $t_V(\mathcal G,mH)=\frac{1}{m}t_V(\mathcal G,H)$.  
\item $t_V(F_Y^*\mathcal G,H) =pt_V(\mathcal G,H)$. 
\item Let $\mathcal E$ be a coherent sheaf on $Y$. 
If there is a morphism $\mathcal G\to \mathcal E$ that is surjective over $V$, then $t_V(\mathcal G,H)\le t_V(\mathcal E,H)$. 
\item Let $\mathcal E$ be a coherent sheaf on $Y$. 
Then $t_V(\mathcal G\otimes \mathcal E,H)\ge t_V(\mathcal G,H)+t_V(\mathcal E, H)$.
\end{enumerate}
\end{lem}
\begin{proof}
The proofs of (2)--(4) are easy. 
We show (1). One can easily check that $T_V(\mathcal G,mH)\subseteq \frac{1}{m}T_V(\mathcal G,H)$, 
so $t_V(\mathcal G,mH)\le \frac{1}{m}t_V(\mathcal G,H)$. 
Take $\varepsilon\in T_V(\mathcal G,H)$. 
Then there is an $e\in\mathbb Z_{>0}$ such that $p^e\varepsilon\in\mathbb Z$ 
and $({F_Y^e}^*\mathcal G)(-p^e\varepsilon H)$ is globally generated over $V$. 
Let $q_e$ and $r_e$ be the quotient and the remainder of the division of 
$p^e\varepsilon$ by $m$. 
Take $g\in\mathbb Z_{>0}$ so that $|p^gr_eH|$ is free 
for each $0\le r_e<m$. 
Then $ -p^{e+g}\varepsilon=-mp^gq_e-p^gr_e, $ so 
$
({F_Y^{e+g}}^*\mathcal G)(-mp^gq_e H)
$
is globally generated over $V$. 
Therefore, 
$$
T_V(\mathcal G,mH) \ni \frac{p^gq_e}{p^{e+g}}=\frac{q_e}{p^e}
\xrightarrow{e\to\infty} \frac{\varepsilon}{m}, 
$$
so we get $t_V(\mathcal G,mH)\ge \frac{\varepsilon}{m}$, 
and hence $t_V(\mathcal G,mH)\ge \frac{1}{m}t_V(\mathcal G,H)$. 
\end{proof}
\begin{prop}[\textup{\cite[Proposition~4.7]{Eji17}}] \label{prop:t-wp}
Let $Y$ be a normal projective variety 
and let $\mathcal G$ be a torsion-free coherent sheaf on $Y$. 
Let $H$ be an ample Cartier divisor on $Y$ 
and let $V$ be a regular open subset of $Y$ 
such that $\mathcal G|_V$ is locally free. 
If $t_V(\mathcal G,H) \ge 0$, then $\mathcal G$ is weakly positive over $V$. 
\end{prop}
\begin{proof}
We may assume that $|H|$ is free. 
Fix an $\alpha\in\mathbb Z_{>0}$. 
Since $t_V(\mathcal G, H)\ge 0$, there is an $\varepsilon\in T_V(\mathcal G,H)$ with $-\alpha^{-1} <\varepsilon$. 
Let $e$ be a positive integer such that 
$
({F_Y^e}^*\mathcal G)(-p^e\varepsilon H)
$
is globally generated over $V$, and then so is 
\begin{align*}
S^{\alpha\beta p^e}\big(({F_Y^e}^*\mathcal G)(-p^e\varepsilon H) \big)
& \cong 
\big(S^{\alpha\beta p^e}({F_Y^e}^*\mathcal G) \big) (-\alpha\beta p^{2e} \varepsilon H)
\\ & \cong 
\big({F_Y^e}^*S^{\alpha\beta p^e}(\mathcal G) \big) (-\alpha\beta p^{2e} \varepsilon H)
\cong 
{F_Y^e}^*\big(S^{\alpha\beta p^e}(\mathcal G) (-\alpha\beta p^e\varepsilon H) \big)
\end{align*}
for each $\beta\in \mathbb Z_{>0}$. 
Then there is a morphism 
$$
\sigma:\bigoplus \mathcal O_Y 
\to 
{F_Y^e}^*\big(S^{\alpha\beta p^e}(\mathcal G) (-\alpha\beta p^e\varepsilon H) \big)
$$
that is surjective over $V$. 
Put $\mathcal F:=({F_Y^e}_*\mathcal O_Y) \otimes ({F_Y^e}_*\mathcal O_Y)^*$. 
Then there is a natural morphism $\tau:\mathcal F \to \mathcal O_Y$ 
that is surjective over $V$ (note that ${F_Y^e}_*\mathcal O_Y$ is locally free over $V$, since $V$ is regular). 
Take an $l\in\mathbb Z_{>0}$ so that $\mathcal F(lH)$ is globally generated. 
Then we obtain the sequence of morphisms
\begin{align*}
\bigoplus \mathcal F(lH)
\cong & ({F_Y^e}_*\bigoplus \mathcal O_Y) \otimes ({F_Y^e}_*\mathcal O_Y)^*(lH)
\\ \xrightarrow{({F_Y^e}_*\sigma)\otimes ({F_Y^e}_*\mathcal O_Y)^*(lH)} &
\left( {F_Y^e}_*
{F_Y^e}^*\big(S^{\alpha\beta p^e}(\mathcal G) (-\alpha\beta p^e\varepsilon H) \big)
\right) \otimes ({F_Y^e}_*\mathcal O_Y)^*(lH)
\\ \cong & 
\big(S^{\alpha\beta p^e}(\mathcal G) (-\alpha\beta p^e\varepsilon H) \big) \otimes \mathcal F(lH)
\\ \xrightarrow{\textup{{\tiny induced by $\tau$}}} &
S^{\alpha\beta p^e}(\mathcal G) ((l-\alpha\beta p^e\varepsilon) H) 
\end{align*}
that are surjective over $V$, 
so $S^{\alpha\beta p^e}(\mathcal G)((l-\alpha\beta p^e\varepsilon)H)$ is globally generated over $V$.
Since $-\alpha \varepsilon <1$ and $|H|$ is free, we see that 
$S^{\alpha\beta p^e}(\mathcal G)((l+\beta p^e)H)$
is globally generated over $V$. 
Choose $\beta$ so that $l\le \beta p^e$. 
Then 
$S^{\alpha\beta p^e}(\mathcal G)(2\beta p^eH)$
is globally generated over $V$. 
Hence, replacing $H$ by $2H$, we can check that $\mathcal G$ is weakly positive over $V$. 
\end{proof}
\begin{thm} \label{thm:wp}
Let $X$ be an equi-dimensional projective $k$-scheme satisfying $S_2$ and $G_1$. 
Let $\Delta$ be an effective $\mathbb Q$-AC divisor. 
Let $Y$ be a regular projective variety 
and let $f:X\to Y$ be a surjective morphism. 
Let $M$ be a nef $\mathbb Q$-Cartier divisor on $X$. 
Let $V$ be an open subset of $Y$. 
Set $U:=f^{-1}(V)$ and $f|_U:U\to V$.  
Suppose the following conditions:
\begin{enumerate}[$(1)$]
\item $K_X+\Delta$ is $\mathbb Q$-Cartier; 
\item $K_X+\Delta+M$ is $f$-ample;
\item $K_U+\Delta|_U$ is a $\mathbb Z_{(p)}$-Cartier divisor; 
\item $U$ is flat over $V$; 
\item $\mathrm{Supp}(\Delta)$ does not contain any component of any fiber over $V$;
\item every fiber over $V$ satisfies $S_2$ and $G_1$;
\item $(X_{\overline y}, \Delta|_{X_{\overline y}})$ is $F$-pure for every $y\in V$, where $X_{\overline y}$ is the geometric fiber over $y$, $\Delta|_{X_{\overline y}}$ is the unique extension of the restriction $\Delta|_{W_{\overline y}}$ to $W_{\overline y}$, and $W$ is the Gorenstein locus of $X$. 
\end{enumerate}
Let $H$ be an ample divisor on $Y$. 
Then there exists an $m_0\in\mathbb Z_{>0}$ such that  
$$
t_V(f_*\mathcal O_X(m(K_{X/Y}+\Delta+M)+N), H) \ge 0
$$
for each $m\ge m_0$ with $m(K_X+\Delta)$ and $mM$ Cartier 
and every nef Cartier divisor $N$ on $X$. 
In particular, $f_*\mathcal O_X(m(K_{X/Y}+\Delta+M)+N)$ is weakly positive over $V$. 
\end{thm}
\begin{proof}
Since $t_V(?,lH)=\frac{1}{l}t_V(?,H)$ for each $l\in \mathbb Z_{>0}$, 
we may assume that $|H|$ and $|-K_Y+H|$ are free. 
Put $L:=K_{X/Y}+\Delta+M$. 
Take an $e\in\mathbb Z_{>0}$.
Consider the morphism $f_{Y^e}:X_{Y^e}\to Y^e$ 
and the divisor $\Delta_{Y^e}$. 
Note that $X_{Y^e}\to X$ is a Gorenstein morphism by 
\cite[Corollary~2']{WITO69}, 
and $X_{Y^e}$ satisfies $S_2$ and $G_1$ by 
\cite[Proposition~1~(i\hspace{-1pt}i)]{RF72}. 
By \cite[Proposition~3.6]{Eji22c}, there is an $m_0\in\mathbb Z_{>0}$ 
independent of $e$ such that 
\begin{align*} \tag{$\ast$}
& {f_{Y^e}}_*\phi_{\left( X_{Y^e},\Delta_{Y^e} \right)}^{(g)}\left(mL_{Y^e}+N_{Y^e}\right):
\\ & {F_Y^g}_* {f_{Y^e}}_* \mathcal O_{X_{Y^e}}(\lceil(1-p^g)(K_{X_{Y^e}}+\Delta_{Y^e}) \rceil
+mp^gL_{Y^e}+p^gN_{Y^e})
\\ & \to {f_{Y^e}}_*\mathcal O_{X_{Y^e}}(mL_{Y^e}+N_{Y^e})
\end{align*}
is surjective over $V^e$ for each $m\ge m_0$ with $i|m$, 
each $g\in\mathbb Z_{>0}$ such that $(p^g-1)\Delta|_U$ is integral,
and every nef Cartier divisor $N$ on $X$. 
Put $\mathcal G_m:=f_*\mathcal O_X(mL+N)$ for each $m\ge m_0$ with 
$m(K_X+\Delta)$ and $mM$ Cartier, and set $n:=\dim Y$. 
Then we have 
\begin{align*}
& {F_Y^e}^*\mathcal G_m(m(n+2)H) 
\\ & \cong {f_{Y^e}}_* \mathcal O_{X_{Y^e}}(mL_{Y^e}+N_{Y^e}) (m(n+2)H)
\\ & = {f_{Y^e}}_* \mathcal O_{X_{Y^e}}(m(K_{X_{Y^e}/Y^e}+\Delta_{Y^e}+M_{Y^e})+N_{Y^e}) (m(n+2)H)
\\ & \cong {f_{Y^e}}_* \mathcal O_{X_{Y^e}}(m(K_{X_{Y^e}}+\Delta_{Y^e}+M_{Y^e})+N_{Y^e}) (m(n+1)H) \otimes \mathcal O_Y(-mK_Y+mH), 
\end{align*}
which is globally generated over $V$ 
by the proof of Theorem~\ref{thm:PS-type} and the choice of $H$. 
Hence, we obtain that $-\frac{m(n+2)}{p^e} \le t_V(\mathcal G_m, H)$. 
Taking $e\to \infty$, we conclude that $0\le t_V(\mathcal G_m, H)$. 
\end{proof}
\begin{rem} \label{rem:surj}
In the proof above, we use the surjectivity over $V$ of 
$$
{f_{Y^e}}_*\phi_{\left(X_{Y^e},\Delta_{Y^e}\right)}^{(g)}\left(mL_{Y^e}+N_{Y^e}\right),
$$
which follows from the surjectivity over $V$ of 
$$
{f_{Y^{e+g}}}_*\phi_{\left(X_{Y^e}/Y^e,\Delta_{Y^e}\right)}^{(g)}
\left(mL_{Y^{e+g}}+N_{Y^{e+g}}\right)
\cong
{F_Y^e}^*\left( {f_{Y^g}}_* \phi_{\left(X/Y,\Delta\right)}^{(g)}
\left(mL_{Y^g}+N_{Y^g}\right)
\right) 
$$
(see the proof of \cite[Proposition~3.6]{Eji22c}). 
This surjectivity over $V$ follows from the $f$-ampleness of $L$ and  
the surjectivity over $U_{V^g}$ of 
$
\phi_{\left(X/Y,\Delta\right)}^{(g)} 
$
(recall that $U:=f^{-1}(V)$). 
To ensure this surjectivity over $U_{V^g}$, 
we need assumptions~(5)--(7) in Theorem~\ref{thm:wp}. 
Therefore, when we have the surjectivity over $U_{V^g}$ of 
$
\phi_{\left(X/Y,\Delta\right)}^{(g)} 
$
for each $g\in\mathbb Z_{>0}$ with $(p^e-1)(K_U+\Delta|_U)$ Cartier, 
the same consequence as that of Theorem~\ref{thm:wp} holds
without assumptions~(5)--(7). 
Note that, when $\Delta=0$, the surjectivity of $\phi_{(X/Y,0)}^{(e)}$ over 
$U_{V^e}$ implies that $f|_U:U\to V$ is an $F$-pure morphism
(\cite[\S 2]{Has10}). 
\end{rem}
\begin{prop} \label{prop:Weil}
Let $(X,\Delta)$ be a normal projective pair and 
let $f:X\to Y$ be a surjective morphism to a 
normal projective variety $Y$. 
Suppose that 
\begin{itemize}
\item $\Delta$ is $\mathbb Z_{(p)}$-Weil, 
\item $(K_X+\Delta)|_{X_{\overline\eta}}$ is $\mathbb Z_{(p)}$-Cartier,  
\item $\left(X_{\overline\eta},\Delta|_{X_{\overline \eta}}\right)$ is $F$-pure, and  
\item $f$ is equi-dimensional, 
\end{itemize}
Let $D$ be a $\mathbb Z_{(p)}$-Weil divisor on $Y$. 
If $-K_{X/Y}-\Delta +f^*D$ is a nef $\mathbb Z_{(p)}$-Cartier divisor, 
then $\mathcal O_Y(mD)$ is pseudo-effective for an $m\in\mathbb Z_{>0}$ 
with $mD$ integral. 
\end{prop}
Note that since $f$ is equi-dimensional, we can define the pullback of 
a $\mathbb Q$-Weil divisor on $Y$. 
\begin{proof}
Let $A$ be a very ample Cartier divisor on $X$. 
Take an $n\in\mathbb Z_{>0}$ with $p\nmid n$. 
Since 
$
-nm(K_{X/Y}+\Delta) +nmf^*D +A
$
is an ample $\mathbb Z_{(p)}$-Cartier divisor on $X$, it is $\mathbb Z_{(p)}$-linearly equivalent to an effective $\mathbb Z_{(p)}$-Cartier divisor $\Gamma$ on $X$ such that 
$
\left(X_{\overline\eta},\left(\Delta+\frac{1}{nm}\Gamma\right)|_{X_{\overline\eta}} \right)
$ 
is $F$-pure by \cite[Corollary~6.10]{SW13}. 
Then \cite[Theorem~5.1]{Eji17} tells us that 
$$
f_*\mathcal O_X(lA+lnmf^*D) 
\cong 
f_*\mathcal O_X\left(lnm\left(K_{X/Y}+\Delta+\frac{1}{nm}\Gamma\right) \right)
$$
is pseudo-effective for each $l\gg0$ divisible enough (in \cite{Eji17}, instead of the pseudo-effectivity, the terminology ``weak positivity'' is used). 
Since $f$ is equi-dimensional, 
$f^*f_*\mathcal O_X(lA+lnmf^*D)$ is also pseudo-effective by \cite[Lemma~2.4~(1)]{EG19}, 
so the natural morphism 
$$
f^*f_*\mathcal O_X(lA+lnmf^*D) \to \mathcal O_X(lA+lnmf^*D)
$$
that is generically surjective 
implies that $\mathcal O_X(A+nmf^*D)$ is pseudo-effective. 
Since $n$ is any positive integer, a standard argument shows that 
$\mathcal O_X(mf^*D)$ is also pseudo-effective. 
Because the pseudo-effectivity of a coherent sheaf is equivalent to that of the restriction of the coherent sheaf to an open subset whose complement has codimension at least two, 
we see that $f^*\mathcal O_X(mD)$ is pseudo-effective, 
so \cite[Lemma~2.4~(2)]{EG19} implies that $\mathcal O_X(mD)$ 
is pseudo-effective. 
\end{proof}
\section{Algebraic fiber spaces with nef relative anti-canonical divisor}
\label{section:nef relative anti}
In this section, we work over a perfect field $k$ of characteristic $p>0$, 
and study an algebraic fiber space with nef relative anti-canonical divisor. 
\begin{thm}[\textup{\cite[Theorem~4.1 and Proposition~4.2]{PZ19}}]
\label{thm:PZ}
Let $(X,\Delta)$ be a normal projective pair 
and let $Y$ be a smooth projective variety. 
Let $f:X\to Y$ be a surjective morphism. 
Suppose that 
\begin{itemize} \leftskip=-10pt
\item $-K_{X/Y}-\Delta$ is a nef $\mathbb Q$-Cartier divisor, and 
\item $\left(X_{\overline\eta},\Delta|_{X_{\overline\eta}}\right)$
is strongly $F$-regular, where $X_{\overline\eta}$ is the geometric generic fiber of $f$. 
\end{itemize}
Then the following hold:
\begin{enumerate}[$(1)$]
\item $f$ is equi-dimensional. In particular, if $X$ is Cohen--Macaulay, then $f$ is flat.   
\item Every geometric fiber of $f$ is reduced. 
\item $\mathrm{Supp}(\Delta)$ does not contain any component of any fiber of $f$. 
\end{enumerate}
\end{thm}
Note that the assumption that $\left(X_{\overline\eta},\Delta|_{X_{\overline\eta}}\right)$ is strongly $F$-regular is equivalent to that of \cite{PZ19} (strong $F$-regularity of general fibers) by \cite[Theorem~B]{PSZ18}.
\begin{thm} \label{thm:num flat}
Let $(X,\Delta)$ be a Cohen--Macaulay projective pair
and let $Y$ be a smooth projective variety.
Let $f:X\to Y$ be a surjective morphism. 
Suppose that 
\begin{itemize} \leftskip=-10pt
\item $-K_{X/Y}-\Delta$ is a nef $\mathbb Z_{(p)}$-Cartier divisor, 
\item $\left(X_{\overline\eta},\Delta|_{X_{\overline\eta}}\right)$
is strongly $F$-regular, where $X_{\overline\eta}$ is the geometric generic fiber of $f$. 
\end{itemize}
Let $B$ be an ample Cartier divisor on $X$. 
Then the following hold: 
\begin{enumerate}[$(1)$]
\item There exists an $n_0\in\mathbb Z_{>0}$ such that 
for every nef $\mathbb Q$-Cartier divisor $N$ on $X$ 
and for each $m,n\in\mathbb Z_{>0}$ with $n \ge n_0$, the sheaf 
$$
f_*\mathcal O_X(m(nB+N)) <-\frac{m}{ir_{i,n}}E_{i,n}>
$$
is a numerically flat $\mathbb Q$-twisted vector bundle, where $i$ is the minimum positive integer such that $iM$ is Cartier, $r_{i,n}:=\mathrm{rank}\,f_*\mathcal O_X(i(nB+N))$ and $E_{i,n}$ is a Cartier divisor on $Y$ such that $\mathcal O_Y(E_{i,n})\cong \det f_*\mathcal O_X(i(nB+N))$. 
\item There exists an $f$-ample Cartier divisor $A$ on $X$ such that 
$$
f_*\mathcal O_X(mA)
$$
is numerically flat for each $m\in \mathbb Z_{\ge 0}$. 
\end{enumerate}
\end{thm}
Note that $f$ is flat, since $f$ is equi-dimensional by Theorem~\ref{thm:PZ}
and $X$ is Cohen--Macaulay. 
To prove the theorem, we prepare the following lemmas:
\begin{lem} \label{lem:wp_isom}
Let $Y$ be a projective variety and let $\mathcal E$ be a vector bundle 
on $Y$ such that $\det \mathcal E$ is numerically trivial. 
Let $\mathcal G$ be a coherent subsheaf of $\mathcal E$. 
If $\mathrm{rank}\,\mathcal G=\mathrm{rank}\,\mathcal E$ 
and $t(\mathcal G, H)\ge 0$ for some ample Cartier divisor $H$ on $Y$, 
then $\mathcal G=\mathcal E$. 
\end{lem}
\begin{proof}
Let $\pi: Y'\to Y$ be a flattening of $\mathcal G$. 
Put $\mathcal G':=(\pi^*\mathcal G)/_{\textup{torsion}}$
and $\mathcal E':=\pi^*\mathcal E$. 
Then there is a natural morphism $\sigma:\mathcal G'\to \mathcal E'$. 
It is enough to show that $\sigma$ is an isomorphism. 
Consider the induced injective morphism 
$\det \sigma:\det\mathcal G' \to \det\mathcal E'$. 
Then $(\det\mathcal G')^*$ is numerically equivalent to an effective divisor. 
Furthermore, one can easily check that $t_{\pi^{-1}(V)}(\det\mathcal G', H')\ge 0$, where $H'$ is an ample Cartier divisor on $Y'$, 
so $\det\mathcal G'$ is weakly positive by Proposition~\ref{prop:t-wp}, 
and hence it is pseudo-effective. 
This means that $\det\mathcal G'$ is numerically trivial and 
$\det\sigma$ is an isomorphism, 
so $\sigma$ is also an isomorphism. 
\end{proof}
\begin{proof}[Proof of Theorem~\ref{thm:num flat}]
First, we prove (1). 
Let $d$ be the minimum positive integer such that $(p^d-1)(K_X+\Delta)$ is Cartier. 
Let $V \subseteq Y$ be the maximal open subset such that 
$\phi_{(X/Y,\Delta)}^{(d)}$ is surjective over $f_{Y^d}^{-1}(V)$. 
Then $V\ne\emptyset$, as $\left(X_{\overline\eta},\Delta|_{X_{\overline\eta}}\right)$ is $F$-pure. 
Also, $\phi_{(X/Y,\Delta)}^{(e)}$ is surjective over $f_{Y^e}^{-1}(V)$ 
for each $e\in\mathbb Z_{>0}$ with $d|e$. 

Let $j$ be a positive integer such that $|jB|$ is free. 
Let $n$ be an integer large enough. 
Then the following properties hold. 
\begin{itemize}
\item $f_*\mathcal O_X(nB+N)$ is locally free for every nef Cartier divisor $N$ on $X$. 
This follows from Keeler's relative Fujita vanishing \cite[Theorem~1.5]{Kee03}, the flatness of $f$ and the cohomology and base change. 
\item $f_*\mathcal O_X(nB+N)$ is globally generated for every nef Cartier divisor $N$ on $X$. 
This follows from Lemma~\ref{lem:gg}. 
\item $|nB+N|$ is free for every nef Cartier divisor $N$ on $X$. 
This follows from the Fujita vanishing theorem and Castelnuovo--Mumford regularity~\cite[Theorem~1.8.5]{Laz04I} (with respect to $jB$). 
\item The natural morphism 
$$
f_*\mathcal O_X(nB+N) \otimes f_*\mathcal O_X((\dim X+1)(nB+N)+N')
\to f_*\mathcal O_X((\dim X+2)(nB+N)+N')
$$
is surjective for every $f$-nef Cartier divisors $N$ and $N'$ on $X$. 
This follows from Keeler's relative Fujita vanishing \cite[Theorem~1.5]{Kee03} and the relative Castelnuovo--Mumford regularity~\cite[Example~1.8.24]{Laz04I} (with respect to $f$ and $jB$).
\item For each $e\in\mathbb Z_{>0}$ with $(p^e-1)(K_X+\Delta)$ Cartier 
and for each $m\in\mathbb Z_{>0}$, 
\begin{align*}
&{f_{Y^e}}_* \phi_{\left(X/Y,\Delta\right)}^{(e)}((nB+N)_{Y^e}) : 
\\ & f^{(e)}_* \mathcal O_{X^e}((1-p^e)(K_{X/Y}+\Delta) +p^em(nB+N))
\to
{F_Y^e}^*f_*\mathcal O_X(m(nB+N))
\end{align*}
is surjective over $V$. This follows from the same argument as that of the proof of \cite[Lemma~3.7]{Eji19p}. 
\end{itemize}
Fix a nef $\mathbb Q$-Cartier divisor $N$ on $X$ and put $P:=nB+N$. 
Set $r_m:=\mathrm{rank}\,f_*\mathcal O_X(mP)$ 
for each $m\in\mathbb Z_{>0}$ with $mP$ Cartier.
Let $E_m$ be a divisor on $Y$ such that $\mathcal O_Y(E_m)\cong \det f_*\mathcal O_X(mP)$ for each $m\in\mathbb Z_{>0}$ with $mP$ Cartier. 
Let $H$ be an ample divisor on $Y$. 
\begin{step} \label{step:t}
In this step, we prove that 
$
t\left(\left(\bigotimes^{r_m}f_*\mathcal O_X(mP)\right)(-E_m), H\right)
\ge 
0
$
for each $m \in\mathbb Z_{>0}$ with $mP$ Cartier. 
Fix an $m\in\mathbb Z_{>0}$. 
We consider the $r_m$-th fiber product 
$$
Z:=\overbrace{X\times_Y\cdots\times_Y X}^{r_m}
$$
of $X$ over $Y$. Let $h:Z\to Y$ be the natural morphism.
Since $f$ is flat by Theorem~\ref{thm:PZ}, so is the $i$-th projection 
$\mathrm{pr}_i:Z\to X$, so we can define the pullback by $\mathrm{pr}_i$
of every $\mathbb Q$-Weil divisor on $X$. 
Set $\Gamma:=\sum_{i=1}^{r_m}\mathrm{pr}_i^*\Delta$ 
and $C:=m\sum_{i=1}^{r_m}\mathrm{pr}_i^*P$. 
Then 
$$
K_{Z/Y} +\Gamma \sim \sum_{i=1}^{r_m} \mathrm{pr}_i^*(K_{X/Y}+\Delta), 
$$
so $-K_{Z/Y}-\Gamma$ is nef. 
By \cite[Lemma~4.1]{Eji22d}, the pair
$
\left(Z_{\overline \eta}, \Gamma|_{Z_{\overline\eta}}\right)
$
is strongly $F$-regular. 
Since 
\begin{align*}
h_*\mathcal O_Z(C-h^*E_m)
\cong h_*\mathcal O_Z(C)(-E_m)
\cong \bigotimes^{r_m} f_*\mathcal O_X(mP)(-E_m)
\end{align*}
and there is an injective morphism 
$$
\mathcal O_Y(E_m)\cong 
\det f_*\mathcal O_X(mP)
\to \bigotimes^{r_m} f_*\mathcal O_X(mP)
$$
that is locally defined as 
$$
x_1\wedge \cdots\wedge x_{r_m} \mapsto 
\sum_{\sigma\in\mathfrak S} \mathrm{sign}(\sigma) x_1\otimes \cdots\otimes x_{r_m}, 
$$
we have $H^0(Y,h_*\mathcal O_Z(C-h^*E_m))\cong H^0(Z,\mathcal O_Z(C-h^*E_m))\ne0$.
Take a $D\in|C-h^*E_m|$. Note that $D$ is $h$-ample. 
Let $l$ be a positive integer with $p\nmid l$ such that 
$l(K_X+\Delta)$ is Cartier and 
$
\left(Z_{\overline \eta}, \Gamma|_{Z_{\overline\eta}}+\frac{1}{l}D|_{Z_{\overline\eta}}\right)
$
is strongly $F$-regular. 
Put $\Gamma':=\Gamma+\frac{1}{l}D$. 
By \cite[Th\'eor\`eme~(12.2.4)]{Gro65} and \cite[Theorem~B]{PSZ18}, 
there is a dense open subset $Y_0$ of $Y$ such that for every $y\in Y_0$, 
\begin{itemize}
\item $Z_y$ is geometrically normal, and 
\item $\left(Z_{\overline y},\Gamma'|_{Z_{\overline y}}\right)$ is strongly $F$-regular, where $Z_{\overline y}$ is the geometric fiber over $y$, $\Gamma'|_{Z_{\overline y}}$ is the unique extension of the restriction $\Gamma'|_{W_{\overline y}}$ to $W_{\overline y}$, and $W$ is the Gorenstein locus $W$ of $Z$. 
\end{itemize}
Put 
$$
M:=-K_{Z/Y}-\Gamma \quad \textup{and} \quad 
N_e:=(1-p^e)(K_{Z/Y}+\Gamma) 
$$ 
for each $e\in\mathbb Z_{>0}$. Then 
\begin{align*}
(1-p^e)(K_{Z/Y}+\Gamma) +p^eD
& \sim lp^e\left(K_{Z/Y}+\Gamma +\frac{1}{l}D +M\right) +N_e
\\ & = lp^e\underbrace{\left(K_{Z/Y}+\Gamma'+M\right)}_{\textup{$h$-ample}} +N_e, 
\end{align*}
so 
$$
t\left( h_*\mathcal O_Z((1-p^e)(K_{Z/Y}+\Gamma)+p^eD), H\right) \ge 0
$$
for $e\gg0$ with $(p^e-1)(K_Z+\Gamma')$ Cartier by Theorem~\ref{thm:wp}. 
Note that $l(K_Z+\Gamma)$ is Cartier by the choice of $l$. 
We now have the following sequence of morphisms 
that are generically surjective: 
\begin{align*} \tag{$\ast1$}
& h_*\mathcal O_Z((1-p^e)(K_{Z/Y}+\Gamma) +p^eD)
\\ & \cong h_*\mathcal O_Z((1-p^e)(K_{Z/Y}+\Gamma) +p^eC -p^eh^*E_m)
\\ & \cong h_*\mathcal O_Z((1-p^e)(K_{Z/Y}+\Gamma) +p^eC)(-p^eE_m)
\\ & \cong \left( \bigotimes^{r_m} f_*\mathcal O_X((1-p^e)(K_{X/Y}+\Delta) +mp^eP) \right) (-p^eE_m)
\\ & \xrightarrow{\bigotimes^{r_m} \left({f_{Y^e}}_*\phi_{(X/Y,\Delta)}^{(e)}(mP_{Y^e})\right)(-p^eE_m)} 
\left(\bigotimes^{r_m}{F_Y^e}^*f_*\mathcal O_X(mP) \right) (-p^eE_m) 
\\ & \cong {F_Y^e}^* \left( \left( \bigotimes^{r_m}f_*\mathcal O_X(mP) \right) (-E_m)\right). 
\end{align*}
Hence, by Lemma~\ref{lem:t-basic}, 
\begin{align*}
&t\left(\left(\bigotimes^{r_m}f_*\mathcal O_X(mP)\right)(-E_m), H\right)
\\ = & \frac{1}{p^e}t\left(
{F_Y^e}^* \left( \left( \bigotimes^{r_m}f_*\mathcal O_X(mP) \right) (-E_m)\right), H
\right)
\\ \ge & \frac{1}{p^e} t\left(
h_*\mathcal O_Z((1-p^e)(K_{Z/Y}+\Gamma) +p^eD), H
\right) 
\\ \ge & 0. 
\end{align*}
\end{step}
\begin{step} \label{step:phi surj}
In this step, we show that the open subset $V$ defined in the beginning of the proof is equal to $Y$. 
Since $\phi_{(X/Y,\Delta)}^{(e)}$ factors $\left(\phi_{(X/Y,\Delta)}^{(e')}\right)_{Y^e}$ 
for $e'\le e$, it is enough to show that $\phi_{(X/Y,\Delta)}^{(e)}$ 
is surjective for $e$ large and divisible enough. 
Let $\mathcal G$ be the image of the composite of $(\ast1)$. 
Then $t(\mathcal G,H)\ge 0$. 
As 
\begin{align*}
\det \left( \left( \bigotimes^{r_m}f_*\mathcal O_X(mP) \right) (-E_m)\right)
& \cong 
\det \left( \bigotimes^{r_m}f_*\mathcal O_X(mP) \right) \left(-r_m^{r_m}E_m\right)
\\ & \cong \mathcal O_Y\left(r_m^{r_m} E_m -r_m^{r_m}E_m\right)
\cong \mathcal O_Y, 
\end{align*}
by Lemma~\ref{lem:wp_isom}, we see that 
$
\mathcal G=
{F_Y^e}^*\left( \left( \bigotimes^{r_m}f_*\mathcal O_X(mP) \right) (-E_m)\right), 
$
so 
$$
\bigotimes^{r_m}\left(
	{f_{Y^e}}_*\phi_{(X/Y,\Delta)}^{(e)}(mP_{Y^e})
\right) (-p^e E_m)
$$
is surjective, which means that 
$
{f_{Y^e}}_*\phi_{(X/Y,\Delta)}^{(e)}(mP_{Y^e})
$
is also surjective. 
By the commutative diagram
$$
\xymatrix{
f_{Y^e}^*{f_{Y^e}}_*\left({F_{X/Y}^{(e)}}_*\mathcal O_X((1-p^e)(K_X+\Delta)+p^emP) \right) \ar@{->>}[r]^-\alpha \ar[d] & 
f_{Y^e}^*{f_{Y^e}}_*\mathcal O_{X_{Y^e}}(mP_{Y^e}) \ar@{->>}[d]  \\
{F_{X/Y}^{(e)}}_*\mathcal O_X((1-p^e)(K_X+\Delta)+p^emP) \ar[r]_-{\phi_{(X/Y,\Delta)}^{(e)}(mP_{Y^e})} & 
\mathcal O_{X_{Y^e}}(mP_{Y^e}), 
}
$$
where $\alpha:=f_{Y^e}^*{f_{Y^e}}_*\phi_{(X/Y,\Delta)}^{(e)}(mP_{Y^e})$, 
we see that $\phi_{(X/Y,\Delta)}^{(e)}(mP_{Y^e})$ is surjective,
or equivalently, $\phi_{(X/Y,\Delta)}^{(e)}$ is surjective, 
and thus $V=Y$. 
\end{step}
\begin{step} \label{step:num equiv}
In this step, we prove that 
$
nr_nE_m \equiv mr_mE_n
$
for each $m,n\in\mathbb Z_{>0}$ such that $mP$ and $nP$ are Cartier. 
Let $s_e$ and $t_e$ be the integers satisfying 
$np^e=s_emr_m+t_e$ and $\dim X+1\le t_e < \dim X+1+mr_m$. 
Note that $t_eP$ is Cartier. 
By the choice of $B$, 
we have the following sequence of surjective morphisms: 
\begin{align*}
& \bigotimes^{s_e}\left( \bigotimes^{r_m}f_*\mathcal O_X(mP) \right)
\otimes 
f_*\mathcal O_X(\underbrace{(1-p^e)(K_{X/Y}+\Delta)}_{\textup{nef}} +t_e P)
\\ \twoheadrightarrow & f_*\mathcal O_X((1-p^e)(K_{X/Y}+\Delta)+np^eP)
\twoheadrightarrow {F_Y^e}^*f_*\mathcal O_X(nP)
\end{align*}
Taking the tensor product with $\mathcal O_Y(-s_eE_m)$, 
we get the morphism 
\begin{align*}
\mathcal G &:= \bigotimes^{s_e}\left(
	\left( \bigotimes^{r_m}f_*\mathcal O_X(mP) \right)(-E_m)
\right)
\otimes 
\underbrace{f_*\mathcal O_X((1-p^e)(K_{X/Y}+\Delta) +t_e P)}_{\textup{globally generated}}
\\ \twoheadrightarrow & \left( {F_Y^e}^*f_*\mathcal O_X(nP) \right) (-s_eE_m)
\end{align*}
that is surjective over $V$. 
By Step~1 and Lemma~\ref{lem:t-basic}, we get $t(\mathcal G,H)\ge 0$, 
so $t(({F_Y^e}^*f_*\mathcal O_X(nP))(-s_eE_m),H)\ge0$.
Since 
\begin{align*}
\det \big(({F_Y^e}^*f_*\mathcal O_X(nP))(-s_eE_m) \big)
& \cong \det \big( {F_Y^e}^*f_*\mathcal O_X(nP) \big) (-r_ns_eE_m)
\\ & \cong \mathcal O_X(p^eE_n-r_ns_eE_m), 
\end{align*}
we see that $t(\mathcal O_X(p^eE_n-r_ns_eE_m),H)\ge0$, 
which means that $p^eE_n-r_ns_eE_m$ is pseudo-effective. 
As 
$$
E_n -\frac{r_ns_e}{p^e}E_m \xrightarrow{e\to\infty} E_n-\frac{nr_n}{mr_m}E_m, 
$$
we obtain that $mr_mE_n-nr_nE_m$ is pseudo-effective. 
Replacing $n$ and $m$, we also get that $nr_nE_m-mr_mE_n$ is pseudo-effective. 
Hence, we see from \cite[Lemma~5.4]{GLPSTZ} that $nr_nE_m\equiv mr_mE_n$.  
\end{step}
\begin{rem}
In positive characteristic, it is not known whether or not the weak positivity of a vector bundle $\mathcal V$ implies the pseudo-effectivity of $\det \mathcal V$, but we know that if $t(\mathcal V, H)\ge 0$ then $\det\mathcal V$ is pseudo-effective, so we use the invariant $t$ in the proof. 
\end{rem}
\begin{step} \label{step:num flat1}
In this step, we prove that 
$
f_*\mathcal O_X(inP)\left(-\frac{n}{r_i}E_i \right)
$ 
is numerically flat for some $n\in\mathbb Z_{>0}$ such that $r_i|n$
and $nP$ is Cartier, where $i$ is the minimum positive integer such that 
$iP$ is Cartier. 
Put $m=i$ and let $Z$, $h:Z\to Y$, $C$, $D$, $\Gamma$, $\Gamma'$ and $l$ be as in Step~\ref{step:t}. 
Since $\left(Z_{\overline \eta}, \Gamma'|_{Z_{\overline \eta}}\right)$
is strongly $F$-regular, there is an $n\in\mathbb Z_{>0}$ 
such that $r_i|n$, $nP$ is Cartier and 
\begin{align*}
{h_{Y^e}}_*\phi_{(Z/Y,\Gamma')}^{(e)}(D_{Y^e}):
h_*\mathcal O_Z((1-p^e)(K_{Z/Y}+\Gamma')+p^enD)
\to 
{F_Y^e}^*h_*\mathcal O_Z(nD)
\end{align*}
is generically surjective. 
Put 
$$
M:=K_{Z/Y}+\Gamma, \quad
N_e:=(1-p^e)(K_{Z/Y}+\Gamma) \quad \textup{and} \quad
m_e:=lnp^e+1-p^e
$$
for each $e\in\mathbb Z_{>0}$.
Then 
\begin{align*}
(1-p^e)(K_{Z/Y}+\Gamma')+p^enD
& \sim 
(1-p^e)(K_{Z/Y}+\Gamma)+(lnp^e+1-p^e)\frac{1}{l}D
\\ & =
N_e +m_e\frac{1}{l}D
\\ & \sim 
m_e\left(K_{Z/Y}+\Gamma+\frac{1}{l}D+M \right)+N_e
\\ & =
m_e(\underbrace{K_{Z/Y}+\Gamma'+M}_{\textup{$h$-ample}} )+N_e, 
\end{align*}
so 
$$
t\left( h_*\mathcal O_Z((1-p^e)(K_{Z/Y}+\Gamma')+p^enD) ,H \right) \ge 0
$$
for $e\gg0$ with $(p^e-1)(K_Z+\Gamma')$ Cartier by Theorem~\ref{thm:wp}. 
Note that $m_e(K_Z+\Gamma')$ is Cartier by the choice of $l$. 
Set $\mathcal G:=\mathrm{Im}\left({h_{Y^e}}_*\phi_{(Z/Y,\Gamma')}^{(e)}(D_{Y^e})\right)$. 
Then $t(\mathcal G,H)\ge 0$.
We show that $\det h_*\mathcal O_Z(nD)$ is numerically trivial. 
If this holds, 
then by an argument similar to that of Step~\ref{step:phi surj}, 
we obtain that $\phi_{(Z/Y,\Gamma')}^{(e)}$ is surjective. 
Thus we see from Remark~\ref{rem:surj} that 
$$
t_Y\left( h_*\mathcal O_Z((1-p^e)(K_{Z/Y}+\Gamma')+p^enD) ,H \right) \ge 0. 
$$
Using the surjectivity of ${h_{Y^e}}_*\phi_{(Z/Y,\Gamma')}^{(e)}(nD_{Y^e})$
again, we get 
$$
t_Y\left( {F_Y^e}^*h_*\mathcal O_Z(nD) ,H \right) \ge 0, 
$$
so ${F_Y^e}^*h_*\mathcal O_Z(nD)$ is nef by Proposition~\ref{prop:t-wp}, 
and hence it is numerically flat, 
or equivalently, 
\begin{align*} \tag{$\ast2$}
h_*\mathcal O_Z(nD)
\cong & h_*\mathcal O_Z(nC) (-nE_i)
\\ & \cong \left( \bigotimes^{r_i} f_*\mathcal O_X(inP) \right) (-nE_i)
\cong \bigotimes^{r_i} \left( f_*\mathcal O_X(inP) \left(-\frac{n}{r_i}E_i\right)\right)
\end{align*}
is numerically flat, so $f_*\mathcal O_X(inP)(-\frac{n}{r_i}E_i)$ 
is also numerically flat. 
Let us return to the proof of the numerical triviality of 
$\det h_*\mathcal O_Z(nD)$. 
We have 
\begin{align*}
\det \left( f_*\mathcal O_X(inP)\left(-\frac{n}{r_i}E_i\right)\right) 
\cong \mathcal O_Y\left(E_{in}-\frac{nr_{in}}{r_i}E_i \right)
\underset{\textup{Step~\ref{step:num equiv}}}{\equiv}
\mathcal O_Y(E_{in}-E_{in})
=\mathcal O_Y. 
\end{align*}
Hence, by $(\ast2)$, we see that 
$\det h_*\mathcal O_Z(nD)$ is numerically trivial. 
\end{step}
\begin{step} \label{step:num flat2}
In this step, we show that 
$ f_*\mathcal O_X(mP)<-\frac{m}{ir_i}E_i> $
is numerically flat for each $m\in\mathbb Z_{>0}$ with $mP$ Cartier. 
Fix an $m\in\mathbb Z_{>0}$ with $mP$ Cartier. Since 
$$
\det \left( f_*\mathcal O_X(mP)<-\frac{m}{ir_i}E_i> \right)
\cong \mathcal O_Y(E_{m})<-\frac{mr_m}{ir_i}E_i> \underset{\textup{Step~\ref{step:num equiv}}}{\equiv}
\mathcal O_Y, 
$$
it is enough to show that 
$ f_*\mathcal O_X(mr_iP)<-\frac{m}{ir_i}E_i> $ is nef. 
Let $n$ be as in Step~\ref{step:num flat1}. 
For each $e\in\mathbb Z_{>0}$, 
let $u_e$ and $v_e$ be integers satisfying $mp^e=inu_e+v_e$ and 
$\dim X+1\le v_e <\dim X+1+in$. 
By the choice of $B$, we have the following surjective morphisms: 
\begin{align*}
& \bigotimes^{u_e} f_*\mathcal O_X(inP) \otimes f_*\mathcal O_X(\underbrace{(1-p^e)(K_{X/Y}+\Delta)}_{\textup{nef}}+v_e P)
\\ & \twoheadrightarrow f_*\mathcal O_X((1-p^e)(K_{X/Y}+\Delta)+mp^eP)
\xrightarrow{{f_{Y^e}}_*\phi_{(X/Y,\Delta)}^{(e)}(mP_{Y^e})}
{F_Y^e}^* f_*\mathcal O_X(mP)
\end{align*}
for each $e\in\mathbb Z_{>0}$ with $(p^e-1)(K_X+\Delta)$ Cartier. 
Taking the tensor product with 
$\mathcal O_Y\left(-\frac{nu_e}{r_i}E_i\right)$, we obtain 
\begin{align*}
& \underbrace{\bigotimes^{u_e}\left(f_*\mathcal O_X(inP)\left(-\frac{n}{r_i}E_i\right)\right)}_{\textup{numerically flat}}
\otimes \underbrace{f_*\mathcal O_X((1-p^e)(K_{X/Y}+\Delta)+v_e P)}_{\textup{globally generated}}
\\ & \twoheadrightarrow \left({F_Y^e}^*f_*\mathcal O_X(mP) \right)\left(-\frac{nu_e}{r_i}E_i\right). 
\end{align*}
This surjective morphism shows that the target is nef, which is equivalent 
to the nefness of the $\mathbb Q$-twisted vector bundle 
$$
f_*\mathcal O_X(mP)<-\frac{nu_e}{p^er_i}E_i>. 
$$
Since $\lim_{e\to\infty}\frac{nu_e}{p^er_i}=\frac{m}{ir_i}$, we see that 
$
f_*\mathcal O_X(mP)<-\frac{m}{ir_i}E_i>
$
is nef. 
\end{step}
Put $A:=ir_iP-f^*E_i$. Then by the above argument, we see that (2) holds 
for $m\ge 1$. 
\begin{step}
Lastly, we prove (2) for $m=0$. 
Let $f:X\xrightarrow{g}W\xrightarrow{h}Y$ be the Stein factorization. 
Then 
$$
-K_{X/W}-\Delta -g^*K_{W/Y} \sim -K_{X/Y}-\Delta
$$
is nef, so Proposition~\ref{prop:Weil} shows that 
$\mathcal O_X(-K_{W/Y})$ is pseudo-effective. 
Since $f$ is separable, so is $h$, and hence $K_{W/Y}$ is linearly equivalent to an effective Weil divisor. 
Thus, we see from \cite[Lemma~4.4]{Eji19w} that $K_{W/Y}\sim 0$, 
so $h$ is \'etale. 
Then the relative Frobenius morphism $F_{W/Y}^{(1)}:W^1 \to W_{Y^1}$ is an 
isomorphism, so the induced morphism 
$$
{F_Y^1}^*h_*\mathcal O_W
\cong {h_{Y^1}}_*\mathcal O_{W_{Y^1}}
\xrightarrow{{h_{Y^1}}_*{F_{W/Y}^{(1)}}^\sharp} 
h_*\mathcal O_W
$$
is an isomorphism. From this isomorphism, one can easily check that 
$h_*\mathcal O_W \cong f_*\mathcal O_X$ is numerically flat. 
\end{step}
\end{proof}
\begin{thm} \label{thm:decomp1}
Let $(X,\Delta)$ be a strongly $F$-regular projective pair 
and let $Y$ be a smooth projective variety. 
Let $f:X\to Y$ be a surjective morphism.
Suppose that 
\begin{itemize} \leftskip=-10pt
\item $-K_{X/Y}-\Delta$ is a nef $\mathbb Z_{(p)}$-Cartier divisor, and 
\item $\left(X_{\overline\eta},\Delta|_{X_{\overline\eta}}\right)$ 
is strongly $F$-regular, where $X_{\overline\eta}$ is the geometric generic fiber of $f$. 
\end{itemize}
Then the following hold. 
\begin{enumerate}[$(1)$]
\item If $Y$ is separably rationally connected, 
then 
$
(X,\Delta)\cong (F,\Gamma)\times_k Y
$ 
as $Y$-schemes. 
\item If $\pi^{\textup{\'et}}(Y)=0$, then there exists an $e\in\mathbb Z_{>0}$ such that 
$
(X_{Y^e},\Delta_{Y^e}) \cong (F, \Gamma) \times_k Y
$ 
as $Y$-schemes. 
\item If $\pi^{\textup{\'et}}(Y)$ is finite, then there exists a finite surjective morphism $\pi:Z\to Y$ that is the composite of a finite \'etale cover and an iterated Frobenius morphism such that 
$
(X_Z, \Delta_Z) \cong (F,\Gamma) \times_k Z
$
as $Z$-schemes. 
\item If $k\subseteq \overline{\mathbb F_p}$, then there exists a finite surjective morphism $\pi:Z\to Y$ that is the composite of a finite \'etale cover and an iterated Frobenius morphism such that 
$
(X_Z, \Delta_Z) \cong (F,\Gamma) \times_k Z
$
as $Z$-schemes. 
\end{enumerate}
Here, $(F,\Gamma)$ in the above statements is $($a disjoint union of copy of$)$ strongly $F$-regular pair. 
\end{thm}
Before starting the proof of the theorem, we give three propositions. 
\begin{prop}[\textup{\cite[Proposition~8.2]{Lan12}}] \label{prop:Langer}
\label{prop:simply conn}
Let $Y$ be a smooth projective variety with $\pi^{\textup{\'et}}(Y)=0$. 
Let $\mathcal E$ be a numerically flat vector bundle of rank $r$ on $Y$. 
Then ${F_Y^e}^*\mathcal E \cong \mathcal O_Y^{\oplus r}$
for some $e\in\mathbb Z_{\ge 0}$. 
\end{prop}
The author learned the following proposition from Shou Yoshikawa. 
\begin{prop} \label{prop:sep rat conn}
Let $Y$ be a smooth projective separably rationally connected variety. 
Then every numerically flat vector bundle is isomorphic to 
a direct sum of structure sheaves $\mathcal O_Y$. 
\end{prop}
\begin{proof}
Let $\mathcal E$ be a numerically flat vector bundle of rank $r$. 
Since $\pi^{\textup{\'et}}(Y)=0$, by Proposition~\ref{prop:Langer}, 
there is an $e \in \mathbb Z_{\ge0}$ such that ${F_Y^e}^*\mathcal E \cong \mathcal O_Y^{\oplus r}$. 
Taking the tensor product of the exact sequence 
$$
0 \to \mathcal O_Y \to {F_Y}_*\mathcal O_Y \to B_Y^1 \to 0
$$
and ${F_Y^{e-1}}^*\mathcal E$, we get the exact sequence 
$$
0 \to H^0(Y, {F_Y^{e-1}}^*\mathcal E) 
\to H^0(Y, {F_Y^e}^*\mathcal E)
\to H^0(Y, {F_Y^e}^*\mathcal E \otimes B_Y^1)
$$
Since ${F_Y^e}^*\mathcal E\cong \mathcal O_Y^{\oplus r}$ and 
$H^0(Y,B_Y^1) \subseteq H^0(Y,\Omega_Y^1)=0$, 
we obtain that 
$$
H^0(Y,{F_Y^{e-1}}^*\mathcal E) 
\cong H^0(Y,{F_Y^e}^*\mathcal E) 
\cong k^{\oplus r}, 
$$
which means that ${F_Y^{e-1}}^*\mathcal E \cong \mathcal O_Y^{\oplus r}$. 
Repeating this argument, we get the assertion.
\end{proof}
\begin{prop}[\textup{\cite[Theorem~1.1]{Lan12}, \cite[Lemma~2.5]{PZ19}}]
	\label{prop:Fp}
Suppose that $k\subseteq \overline{\mathbb F_p}$. 
Let $\mathcal E$ be a numerically flat vector bundle of rank $r$ 
on a normal projective variety $Y$. 
Then there exists a finite surjective morphism $\pi:Z\to Y$ that is the 
composite of a finite \'etale cover and an iterated Frobenius morphism 
such that $\pi^*\mathcal E\cong \mathcal O_Z^{\oplus r}$. 
\end{prop}
\begin{proof}[Proof of Theorem~\ref{thm:decomp1}]
Let $A$ be an $f$-ample Cartier divisor on $X$ as in Theorem~\ref{thm:num flat}. We may assume that the $\mathcal O_Y$-algebra $\bigoplus_{m\ge0}f_*\mathcal O_X(mA)$ is generated by $f_*\mathcal O_X(A)$. 
We first prove (1). 
By Proposition~\ref{prop:sep rat conn}, $f_*\mathcal O_X(mA) \cong \mathcal O_Y^{\oplus r_m}$ for each $m\in\mathbb Z_{>0}$, where $r_m:=f_*\mathcal O_X(mA)$. 
Hence, $\bigoplus_{m\ge 0}f_*\mathcal O_X(mA)$ comes from $\mathrm{Spec}\,k$, 
and so $X\cong F\times Y$ for a projective variety $F$. 
The strongly $F$-regularity of $F$ follows from that of the general fiber of 
$f$. 
The equality $\Delta=\mathrm{pr}_2^*\Gamma$ follows from 
\cite[Lemma~8.4]{PZ19}. 
Next, we show (2). Let $A$ be as above. 
By Proposition~\ref{prop:simply conn}, there is an $e\in\mathbb Z_{>0}$
such that 
$$
{f_{Y^e}}_*\mathcal O_{X_{Y^e}}(A_{Y^e})
\cong {F_Y^e}^*f_*\mathcal O_X(A)
\cong \mathcal O_Y^{\oplus r_1}. 
$$
Since $\bigoplus_{m\ge0} {f_{Y^e}}_*\mathcal O_{X_{Y^e}}(mA_{Y^e})$ 
is generated by ${f_{Y^e}}_*\mathcal O_{X_{Y^e}}(A_{Y^e})$, 
we see that ${f_{Y^e}}_*\mathcal O_{X_{Y^e}}(mA_{Y^e})$ 
is globally generated, and so 
${f_{Y^e}}_*\mathcal O_{X_{Y^e}}(mA_{Y^e}) \cong \mathcal O_{Y}^{\oplus r_m}$ 
for each $m\in\mathbb Z_{>0}$. 
Thus the assertion follows from the same argument as that of the proof of (1). 
Statement~(3) is proved by taking the base change along a finite \'etale cover $\sigma:Y'\to Y$ with $\pi^{\textup{\'et}}(Y')=0$ and applying (2).  

Combining Proposition~\ref{prop:Fp} and the argument of the proof of (2), 
we can prove~(4). 
\end{proof}
\section{Algebraic fiber spaces over abelian varieties}
\label{section:over abelian}
In this section, we consider algebraic fiber spaces over abelian varieties 
whose anti-canonical divisors satisfy some positivity conditions. 
We work over a perfect field $k$ of characteristic $p>0$. 
\subsection{Numerically flat vector bundles on abelian varieties}
The following theorem follows from Langer's argument~\cite{Lan12}.
\begin{thm}[\textup{\cite[Corollary~5.5]{Lan12}}] \label{thm:filtration}
Let $\mathcal E$ be a numerically flat vector bundle 
on an abelian variety $Y$.  
Then there exists a filtration 
$$
0=\mathcal E_0 \subset \mathcal E_1 \subset \cdots \subset 
\mathcal E_{n-1} \subset \mathcal E_n=\mathcal E
$$
such that $\mathcal E_{i}/\mathcal E_{i-1}$ is an algebraically trivial  
line bundle for each $i=1,\ldots,n$. 
\end{thm}
\begin{proof}
It is enough to show that if a numerically flat vector bundle $\mathcal E$ 
on $Y$ is irreducible (i.e., $\mathcal E$ does not contain any proper 
numerically flat subbundles), 
then $\mathrm{rank}\,\mathcal E=1$. 
Let $m:Y\times_k Y \to Y$ be the morphism of operation. 
Put 
$
\mathcal F
:=\mathcal Hom\left(\mathrm{pr}_2^*\mathcal E, m^*\mathcal E\right). 
$
Then $\mathcal F$ is numerically flat. 
Since $\mathcal E$ is irreducible, we have
$$
H^0\left(\{0\}\times_k Y, \mathcal F|_{\{0\}\times_k Y} \right)
\cong \mathrm{Hom}(\mathcal E,\mathcal E) 
\cong k, 
$$ 
so by Proposition~\ref{prop:locally free} below we see that 
$\mathcal L:={\mathrm{pr}_1}_*\mathcal F$ is a line bundle. 
Consider the following morphisms:
$$
\varphi:
\mathrm{pr}_1^* \mathcal L
\otimes \mathrm{pr}_2^* \mathcal E
\to 
\mathcal F 
\otimes \mathrm{pr}_2^* \mathcal E
\to 
m^* \mathcal E. 
$$
Then the composite $\varphi$ is surjective, as it is surjective 
over $\{y\}\times Y$ for every closed point $y\in Y$. 
Note that one can check that $(m^*\mathcal E)|_{\{y\}\times_k Y}$ 
is also irreducible. 
Since $\varphi$ is a surjective morphism between vector bundles of
the same rank, it is an isomorphism. 
Thus
$$
\mathcal E 
\cong (m^*\mathcal E)|_{Y\times_k\{0\}}
\cong (\mathrm{pr}_1^* \mathcal L
\otimes \mathrm{pr}_2^* \mathcal E)|_{Y\times_k\{0\}}
\cong \mathcal L^{\oplus \mathrm{rank}\,\mathcal E}, 
$$
so $\mathrm{rank}\,\mathcal E=1$. 
\end{proof}
\begin{prop}[\textup{\cite[Corollary~3.2]{Lan12}}] \label{prop:locally free}
Let $X$ be a normal complete variety and let $Y$ be a complete variety. 
Let $\mathcal E$ be a numerically flat vector bundle on $X\times_k Y$. 
Then $h^0(X\times_k\{y\}, \mathcal E|_{X\times_k \{y\}})$ is constant on 
closed points of $Y$.
In particular, ${\mathrm{pr}_1}_*\mathcal E$ is locally free. 
\end{prop}
\begin{defn}[\textup{Unipotent vector bundles}] \label{defn:unipotent}
Let $\mathcal E$ be a vector bundle on an abelian variety $Y$. 
We say that $\mathcal E$ is \textit{unipotent} if $\mathcal E$ is obtained 
as an iterated extension of $\mathcal O_Y$. 
\end{defn}
\begin{thm} \label{thm:num flat ab var}
Let $\mathcal E$ be a vector bundle on an abelian variety $Y$. 
Then the following are equivalent:
\begin{enumerate}[$(1)$]
\item $\mathcal E$ is numerically flat; 
\item $\mathcal E$ is obtained as an iterated extension of algebraically trivial line bundles; 
\item $\mathcal E \cong \bigoplus_i \mathcal U_i \otimes \mathcal L_i$, where
each $\mathcal U_i$ is a unipotent vector bundle and each $\mathcal L_i$ 
is an algebraically trivial line bundle. 
\end{enumerate}
\end{thm}
\begin{proof}
(1) $\Rightarrow$ (2) follows from Theorem~\ref{thm:filtration}. 
(2) $\Rightarrow$ (3) is proved by induction on the rank and the fact that 
for every algebraically trivial line bundles $\mathcal L$ and $\mathcal M$, 
$$
\mathrm{Ext}^1(\mathcal L, \mathcal M) 
\cong 
\begin{cases}
	k^{\oplus \dim Y} & \textup{if $\mathcal L\cong\mathcal M$}, \\
	0 & \textup{otherwise.}
\end{cases}
$$ 
(3) $\Rightarrow$ (1) is obvious. 
\end{proof}
\begin{lem} \label{lem:decomp0}
Let $\mathcal E$ be a numerically flat vector bundle on an abelian variety $Y$. 
Then there exists an isogeny $\pi:Z\to Y$ such that 
$\pi^*\mathcal E$ is a direct sum of algebraically trivial line bundles. 
\end{lem}
\begin{proof}
By Theorem~\ref{thm:num flat ab var}, we may assume that $\mathcal E$ is unipotent. 
Then the claim follows from the well-known fact that, 
in positive characteristic, there is an isogeny
$\pi:Z\to Y$ such that the induced map 
$
\pi^*:H^1(Y,\mathcal O_Y)\to H^1(Z,\mathcal O_Z)
$
is the zero-map. 
\end{proof}
\begin{lem} \label{lem:decomp}
Let $\mathcal E$ be a numerically flat vector bundle on an abelian variety $Y$
and let $\bigoplus_i \mathcal N_i \twoheadrightarrow \mathcal E$
be a surjective morphism from a direct sum of algebraically trivial line bundles on $Y$. 
Then $\mathcal E$ is isomorphic to a direct sum of algebraically trivial 
line bundles on $Y$. 
\end{lem}
\begin{proof}
By Theorem~\ref{thm:num flat ab var}, 
we have $\mathcal E\cong \bigoplus_i\mathcal U_i\otimes \mathcal L_i$, 
where $\mathcal U_i$ is a unipotent vector bundle and $\mathcal L_i$ 
is an algebraically trivial line bundle. 
We may assume that $\mathcal E \cong \mathcal U_1 \otimes \mathcal L_1$. 
Since 
$$
\mathrm{Hom}(\mathcal N_i,\mathcal L_1)
\cong 
\begin{cases} 
k & \textup{if}~\mathcal N_i \cong \mathcal L_1, \\
0 & \textup{otherwise.}
\end{cases}
$$
we may assume that $\mathcal N_i \cong \mathcal L_1$ for each $i$. 
Then $\mathcal U_1$ is globally generated,
so $\mathcal U_1 \cong \mathcal O_Y^{\oplus r}$, where $r:=\mathrm{rank}\,\mathcal U_1$, 
and hence $\mathcal E \cong \mathcal L_1^{\oplus r}$. 
\end{proof}
\subsection{Case when $-K_X$ is nef}
\begin{thm} \label{thm:decomp2}
Let $(X,\Delta)$ be a strongly $F$-regular projective variety 
and let $Y$ be an abelian variety. 
Let $f:X\to Y$ be an algebraic fiber space. 
Suppose that 
\begin{itemize} \leftskip=-10pt
\item $-K_X-\Delta$ is a nef $\mathbb Z_{(p)}$-Cartier divisor, and 
\item $\left(X_{\overline\eta}, \Delta|_{X_{\overline\eta}}\right)$
is strongly $F$-regular, where $X_{\overline\eta}$ is the geometric generic fiber of $f$. 
\end{itemize}
Then $X_y\cong X_z$ for every $y,z\in Y(k)$. 
\end{thm}
\begin{proof}
Let $A$ be as in Theorem~\ref{thm:num flat}. 
We may assume that 
\begin{itemize}
\item the $\mathcal O_Y$-algebra $\bigoplus_{m\ge0}f_*\mathcal O_X(mA)$ is generated by $f_*\mathcal O_X(A)$, and 
\item for each $m\in\mathbb Z_{>0}$ and every closed point $y\in A$, 
	the natural morphism 
$$
f_*\mathcal O_X(mA) \otimes k(y) 
\to
H^0\left(X_y, \mathcal O_{X_y}(mA_y) \right)
$$
is an isomorphism. 
\end{itemize}
By Lemma~\ref{lem:decomp0}, 
replacing $f_*\mathcal O_X(A)$ by its pullback by an isogeny, 
we may assume that $f_*\mathcal O_X(A)$ is isomorphic to a direct sum of 
algebraically trivial line bundles. 
By the surjective morphism 
$
\bigotimes^m f_*\mathcal O_X(A)
\twoheadrightarrow f_*\mathcal O_X(mA)
$
and Lemma~\ref{lem:decomp}, we see that $f_*\mathcal O_X(mA)$ is isomorphic 
to a direct sum of algebraically trivial line bundles. 
Put $\mathcal G:=\bigoplus_{\mathcal L\in\mathrm{Pic}^0(Y)}\mathcal L$. 
Note that $\mathcal G$ is not coherent. 
For each $l,m\in\mathbb Z_{\ge0}$, we have the morphism
$$
\left( f_*\mathcal O_X(lA) \otimes \mathcal G \right)
\otimes \left( f_*\mathcal O_X(mA) \otimes \mathcal G \right)
\to f_*\mathcal O_X((l+m)A) \otimes \mathcal G 
$$
so $\bigoplus_{m\ge0} f_*\mathcal O_X(mA)\otimes\mathcal G$ forms $\mathcal O_Y$-algebra, which induces $k$-algebra 
$$
\bigoplus_{m\ge0} R_m
:= \bigoplus_{m\ge0} H^0(Y, f_*\mathcal O_X(mA)\otimes \mathcal G). 
$$
One can easily check that 
$$
R_m = H^0(Y, f_*\mathcal O_X(mA)\otimes \mathcal G )
\to 
f_*\mathcal O_X(mA) \otimes k(y)
$$
is an isomorphism, so 
$$
\bigoplus_{m\ge0}H^0\left(X_y, \mathcal O_{X_y}\left(mA_y\right)\right)
\cong 
\bigoplus_{m\ge0} R_m 
\cong 
\bigoplus_{m\ge0}H^0\left(X_z, \mathcal O_{X_z}\left(mA_z\right)\right) 
$$
for every $y,z\in A(k)$, which means that $X_y\cong X_z$.  
\end{proof}
\subsection{Case when $-K_X$ is numerically equivalent to a semi-ample divisor}
\begin{thm} \label{thm:decomp3}
Let $(X,\Delta)$ be a strongly $F$-regular projective variety 
and let $Y$ be an abelian variety. 
Let $f:X\to Y$ be a surjective morphism. 
Suppose that 
\begin{itemize} \leftskip=-10pt
\item $K_X+\Delta$ is $\mathbb Z_{(p)}$-Cartier, 
\item $-K_X-\Delta \sim_{\mathbb Q} D+f^*P$ for 
a semi-ample $\mathbb Q$-Cartier divisor $D$ on $X$ and 
a numerically trivial $\mathbb Q$-divisor $P$ on $Y$, and 
\item $\left(X_{\overline\eta}, \Delta|_{X_{\overline\eta}}\right)$
is strongly $F$-regular, where $X_{\overline\eta}$ is the geometric generic fiber of $f$. 
\end{itemize}
Then $P\sim_{\mathbb Q}0$ and there exists an isogeny $\pi:Z\to Y$ such that 
$$
(X_Z, \Delta_Z) \cong (X_0, \Delta|_{X_0}) \times_k Z, 
$$
as $Z$-schemes, where $X_0$ is the fiber over the identity element $0\in Y(k)$. 
\end{thm}
The following lemma is used in the proof of Theorem~\ref{thm:decomp3}. 
\begin{lem} \label{lem:semi-ample decomp}
Let $f:X\to Y$ be a surjective morphism between projective varieties. 
Let $D$ be a Cartier divisor on $X$ with $|D|$ free. 
Let $A$ be an $f$-ample Cartier divisor on $X$. 
Then there exists an $m_0\in\mathbb Z_{>0}$ with the following property:
For each $n\in\mathbb Z_{\ge 0}$, there exists a surjective morphism 
$$
\bigoplus f_*\mathcal O_X(mA+N) 
\twoheadrightarrow f_*\mathcal O_X(nD+mA+N) 
$$
for each $m\ge m_0$ and every nef Cartier divisor $N$ on $X$. 
\end{lem}
\begin{proof}
Since $D$ is nef, it is enough to show the case of $n=1$. 
By the assumption, there is a surjective morphism 
$
\bigoplus \mathcal O_X \twoheadrightarrow \mathcal O_X(D). 
$
Let $\mathcal G$ be the kernel of the above morphism. 
Then, by Keeler's relative Fujita vanishing \cite[Theorem~1.5]{Kee03}, 
there is an $m_0\in\mathbb Z_{>0}$ such that 
$
R^1f_*\mathcal G(mA+N)=0
$
for each $m\ge m_0$ and every nef Cartier divisor $N$ on $X$. 
We then see that the induced morphism 
$$
\bigoplus f_*\mathcal O_X(mA+N)
\to 
f_*\mathcal O_X(D+mA+N)
$$
is surjective for each $m\ge m_0$ and every nef Cartier divisor $N$ on $X$. 
\end{proof}
\begin{proof}[Proof of Theorem~\ref{thm:decomp3}]
Let $B'$ be an ample Cartier divisor on $X$. 
Let $n$ be a positive integer such that $nB'$ satisfies several conditions 
in the beginning of the proof of Theorem~\ref{thm:num flat}. 
Put $B:=nB'$. 
Set $r:=\mathrm{rank}\,f_*\mathcal O_X(B)$ and 
let $E$ be a divisor on $Y$ such that 
$
\mathcal O_Y(E)\cong \det f_*\mathcal O_X(B). 
$
Replacing $f:X\to Y$ be its base change along $[2r]_Y:Y\to Y$, 
we may assume that there is a divisor $E'$ such that $rE'\sim E$. 
Here, $[2r]_Y$ is the morphism defined by $[2r]_Y(y)=2r\cdot y$. 
Then by Theorem~\ref{thm:num flat}, we see that 
$
f_*\mathcal O_X(mB)(-mE')
$ 
is numerically flat for each $m\in\mathbb Z_{>0}$. 
\begin{cl} \label{cl:nef}
The vector bundle $f_*\mathcal O_X(mB+N)(-mE')$ is nef for 
each $m\in\mathbb Z_{>0}$ and every nef Cartier divisor $N$ on $X$. 
\end{cl}
\begin{proof}[Proof of Claim~\ref{cl:nef}]
For each $e\in\mathbb Z_{>0}$ with $(p^e-1)(K_X+\Delta)$ Cartier, 
we have the following morphisms: 
\begin{align*}
& \left(\bigotimes^{p^e-1} f_*\mathcal O_X(mB) \right)
\otimes \underbrace{f_*\mathcal O_X((1-p^e)(K_X+\Delta)+mB+p^eN)}_{\textup{globally generated}}
\\ & \twoheadrightarrow 
f_*\mathcal O_X((1-p^e)(K_X+\Delta) +p^e(mB+N))
\\ & \twoheadrightarrow 
{F_Y^e}^*f_*\mathcal O_X(mB+N).
\end{align*}
Note that the surjectivity of the second morphism follows from Step~\ref{step:phi surj} of the proof of Theorem~\ref{thm:num flat}. 
Since $f_*\mathcal O_X(mB)(-mE')$ is numerically flat, we see that 
$$
\left({F_Y^e}^*f_*\mathcal O_X(mB+N) \right)(-m(p^e-1)E')
$$
is nef, which is equivalent to the nefness of the $\mathbb Q$-twisted 
vector bundle 
$$
f_*\mathcal O_X(mB+N) <-\frac{m(p^e-1)}{p^e}E'>. 
$$
Taking $e\to \infty$, we obtain the claim. 
\end{proof}
Put $A:=B-f^*E'$. 
Since the $\mathcal O_Y$-algebra 
$\bigoplus_{m\ge0} f_*\mathcal O_X(mA)$ is finitely generated, 
by the same argument as that in the proof of Theorem~\ref{thm:decomp2}, 
we may assume that $f_*\mathcal O_X(mA)$ is a direct sum of 
algebraically trivial line bundles on $Y$ for each $m\in\mathbb Z_{>0}$. 
Let $L$ be a numerically trivial divisor on $Y$ 
such that $\mathcal O_Y(L)$ is isomorphic to 
a direct summand of $f_*\mathcal O_X(A)$. 
Replacing $A$ by $A-f^*L$, we may assume that $|A|\ne\emptyset$, 
so we may further assume that $A$ is effective. 
Take an $l\in\mathbb Z_{>0}$ so that $p\nmid l$ and 
$
\left(X_{\overline\eta}, \left(\Delta+\frac{1}{l}A\right)|_{X_{\overline\eta}}\right)
$
is strongly $F$-regular. Put $\Delta':=\Delta+\frac{1}{l}A$. 
Let $m$ be an integer large enough such that $p\nmid m$ and 
\begin{align*} \tag{$\ast$}
& {f_{Y^e}}_*\phi_{(X,\Delta')}^{(e)}(mA_{Y^e}+N_{Y^e}):
\\ & f_*\mathcal O_X((1-p^e)(K_X+\Delta') +p^e(mA+N))
\to
{F_Y^e}^* f_*\mathcal O_X(mA+N)
\end{align*}
is generically surjective for each $e\in\mathbb Z_{>0}$ with $(p^e-1)(K_X+\Delta)$ Cartier and every nef Cartier divisor $N$ on $X$. 
Let $i$ be the minimum positive integer such that 
\begin{itemize}
\item $iD$ and $iP$ are Cartier, 
\item $|iD|$ is free and 
\item $-i(K_X+\Delta)\sim_{\mathbb Z} iD+if^*P$. 
\end{itemize}
Put 
$$
I:=\{r\in\mathbb Z_{\ge 0} | \textup{$0\le r<i$ and $r(K_X+\Delta)$ is Cartier}\}. 
$$
For each $r\in I$, by Claim~\ref{cl:nef}, we see that 
$$
\mathcal G_r:=f_*\mathcal O_X(mA-r(K_X+\Delta))
\cong f_*\mathcal O_X(mB\underbrace{-r(K_X+\Delta)-mf^*L}_{\textup{nef}})
(-mE')
$$
is nef. Set 
$$
\mathbb L(\mathcal G_r):=\{\mathcal L \in \mathrm{Pic}^0(Y)| \textup{there is a generically surjective morphism $\mathcal G_r\to \mathcal L$}\}. 
$$
Then $\mathbb L(\mathcal G_r)$ is a finite set by \cite[Lemma~4.5]{Eji22d}. 
Define 
\begin{align*}
M(1)&:=\bigcup_{r\in I} \mathbb L(\mathcal G_r) 
=\{\mathcal L_1, \ldots, \mathcal L_\lambda\}~\textup{and}
\\ M(n) &:= \left\{\mathcal L_1^{n_1} \otimes \cdots \otimes \mathcal L_\lambda^{n_\lambda} \middle| \textup{$(n_1,\ldots,n_\lambda)\in(\mathbb Z_{\ge0})^{\times\lambda}$ with $\sum_{i=1}^\lambda n_i =n$ } \right\} 
\subseteq \mathrm{Pic}^0(Y) 
\end{align*}
for each $n\in\mathbb Z_{>0}$. 
Note that $M(n)\subseteq M(n+1)$, 
as $\mathcal O_Y\in \mathbb L(\mathcal G_0) \subseteq M(1)$.
Fix an $e \gg 0$ with $lm|(p^e-1)$ 
such that $(p^e-1)(K_X+\Delta)$ is Cartier.  
Put $\nu:=p^e+\frac{1-p^e}{lm}\in\mathbb Z_{>0}$. 
Then $\nu<p^e$. 
Take an $\mathcal L\in M(1)$. Then $\mathcal L\in\mathbb L(\mathcal G_r)$ 
for some $r\in I$. 
Let $q$ and $r'$ be the quotient and the remainder of the division of 
$(r+1)p^e-1$ by $i$. 
Then $r'\in I$ and  
\begin{align*}
& (1-p^e)(K_X+\Delta') +mp^eA -p^er(K_X+\Delta)
\\ & = (1-p^e)(K_X+\Delta)+ \left(\frac{1-p^e}{l}+mp^e\right)A -p^er(K_X+\Delta)
\\ & = (1-(r+1)p^e)(K_X+\Delta)+ \nu mA, 
\\ & = (-r'-qi)(K_X+\Delta)+ \nu mA, 
\\ & \sim_{\mathbb Z} -r'(K_X+\Delta) +iqD +iqf^*P +\nu mA, 
\end{align*}
so we see from Lemma~\ref{lem:semi-ample decomp} that 
there is a surjective morphism 
\begin{align*}
& \bigoplus f_*\mathcal O_X(-r'(K_X+\Delta) +iqf^*P +\nu mA)
\\ & \twoheadrightarrow 
f_*\mathcal O_X((1-p^e)(K_X+\Delta')+mp^eA-p^er(K_X+\Delta)). 
\end{align*}
Using $(\ast)$ for $N=-r(K_X+\Delta)+iqf^*P$, 
we get a generically surjective morphism 
$$
f_*\mathcal O_X(-r'(K_X+\Delta) +iqf^*P +\nu mA) \to \mathcal L^{p^e}.  
$$
Considering the generically surjective morphism 
\begin{align*}
\left( \bigotimes^{\nu-1} \mathcal G_0 \right) 
\otimes \mathcal G_{r'}(iqP)
& \cong \left( \bigotimes^{\nu-1} f_*\mathcal O_X(mA) \right) 
\otimes f_*\mathcal O_X(mA-r'(K_X+\Delta) +iqf^*P)
\\ & \twoheadrightarrow f_*\mathcal O_X(-r'(K_X+\Delta) +iqf^*P+\nu mA) 
\to \mathcal L^{p^e}
\end{align*}
we see that $\mathcal L^{p^e} \otimes \mathcal M^{-1} (-iqP)
\in \mathbb L(\mathcal G_{r'})$ for some $\mathcal M\in M(\nu-1)$,  
which implies that $\mathcal L^{p^e}(-iqP) \in M(\nu)$. 
Note that $\mathcal G_0$ is isomorphic to a direct sum of 
elements of $\mathbb L(\mathcal G_0)$. 
Then, by the pigeonhole principle, for each $n\ge (p^e-1)c+1$, 
where $c:=|\mathbb L(\mathcal G_0)|$, we see that 
$$
\mathcal M \in M(n) 
~\Rightarrow~
\mathcal M(-iqP) \in M(n-p^e+\nu)
\subseteq M(n). 
$$
Since $M(n)$ is finite, there are integers $\alpha>\beta>0$ 
and an $\mathcal M\in M(n)$ such that 
$
\mathcal M(-\alpha iqP) \cong \mathcal M(-\beta iqP), 
$
so $P\sim_{\mathbb Q}0$. 
Let $j>0$ be an integer such that $jP\sim_{\mathbb Z} 0$. 
Then, for each $\mathcal M \in M(n)$, we have 
$$
\mathcal M \cong \mathcal M(-ijqP) \in M(n-p^e+\nu), 
$$
which means that $M(n) \subseteq M(n-p^e+\nu)$, so we get 
$$
\bigcup_{n\ge 1} M(n)
\subseteq \bigcup_{1 \le n \le (p^e-1)c} M(n).
$$ 
Thus, $\bigcup_{n\ge 1} M(n)$ is a finite set, 
so all the elements of $M(1)$ is a torsion line bundle.  
This implies that $\mathcal G_0$ is a direct sum of torsion line bundles, 
so there is an isogeny $\pi:Z\to Y$ such that 
$$
\pi^*\mathcal G_0
= \pi^*f_*\mathcal O_X(mA)
\cong \mathcal O_Z^{\oplus r}, 
$$
and hence the assertion follows from an argument similar to that of the proof of Theorem~\ref{thm:decomp1}. 
\end{proof}
\section{Algebraic fiber spaces over elliptic curves}
\label{section:over elliptic}
In this section, we deal with algebraic fiber spaces over elliptic curves 
over a perfect field $k$ of characteristic $p>0$ 
in the case when the geometric generic fiber is not necessarily 
strongly $F$-regular. 
\begin{thm} \label{thm:decomp ell}
Let $(X,\Delta)$ be a strongly $F$-regular projective pair 
and let $Y$ be an elliptic curve. 
Let $f:X\to Y$ be a surjective morphism. 
Suppose that 
\begin{itemize}
\item $K_X+\Delta$ is $\mathbb Z_{(p)}$-Cartier, and 
\item $-K_X-\Delta\sim_{\mathbb Q}D+f^*P$ for a semi-ample $\mathbb Q$-Cartier divisor $D$ on $X$ and a numerically trivial $\mathbb Q$-divisor $P$ on $Y$. 
\end{itemize}
Then there exists an isogeny $\pi:Z\to Y$ such that 
$(X_Z, \Delta_Z) \cong \left(X_0,\Delta|_{X_0}\right)\times_k Z$ 
as $Z$-schemes, where $X_0$ is the fiber of $f$ over the identity element 
$0\in Y(k)$. 
\end{thm}
Note that $X_0$ is not necessarily reduced. 
\begin{proof}
By \cite[Theorem~7.5]{EP23}, we have an $f$-ample Cartier divisor $L$ on $X$ 
such that $f_*\mathcal O_X(mL)$ is numerically flat 
for each $m\in\mathbb Z_{>0}$. 
Furthermore, by the proof of \cite[Theorem~7.5]{EP23}, 
we can choose $L$ so that $f_*\mathcal O_X(mL+N)$ is nef 
for each $m\in\mathbb Z_{>0}$ and a nef Cartier divisor $N$ on $X$. 
By \cite[p.60]{Oda71}, there is a divisor $G$ on $Y$ of degree zero 
such that $f_*\mathcal O_X(L)(-G)$ has a non-zero global section. 
Replacing $L$ by $L-f^*G$, we may assume that $|L|\ne\emptyset$, 
and we may further assume that $L$ is effective. 
Take an $l\in\mathbb Z_{>0}$ so that $p\nmid l$ and 
$\left(X,\Delta+\frac{1}{l}L \right)$ is strongly $F$-regular. 
Put $\Delta':=\Delta+\frac{1}{l}L$. 
Let $m$ be an integer large enough such that $p\nmid m$ and 
\begin{align*} \tag{$\ast$}
& f_*\phi_{(X,\Delta')}^{(e)}(mL+N): 
\\ & {F_Y^e}_*f_*\mathcal O_X((1-p^e)(K_X+\Delta')+p^e(mA+N))
\to f_*\mathcal O_X(mA+N)
\end{align*}
is surjective for each $e\in\mathbb Z_{>0}$ with $(p^e-1)(K_X+\Delta)$ 
Cartier and every nef Cartier divisor $N$ on $X$. 
Let $i$ be the minimum positive integer such that 
\begin{itemize}
\item $iD$ and $iP$ are Cartier, 
\item $|iD|$ is free and 
\item $-i(K_X+\Delta)\sim_{\mathbb Z}iD+if^*P$. 
\end{itemize}
Put 
$$
I:=\{r\in\mathbb Z_{\ge0} |
\textup{$0\le r<i$ and $r(K_X+\Delta)$ is Cartier} \}
$$
Then 
$$
\mathcal G_r := f_*\mathcal O_X(mA-r(K_X+\Delta))
$$
is nef for each $r\in I$. Set 
$$
\mathbb L(\mathcal G_r) 
:= \{\mathcal L\in\mathrm{Pic}^0(Y)|
\textup{there is a generically surjective morphism 
$\mathcal G_r \to \mathcal L$} \}. 
$$
Then $\mathbb L(\mathcal G_r)$ is a finite set by \cite[Lemma~4.5]{Eji22d}. 
Define 
\begin{align*}
M(1)&:=\bigcup_{r\in I} \mathbb L(\mathcal G_r) 
=\{\mathcal L_1, \ldots, \mathcal L_\lambda\}~\textup{and}
\\ M(n) &:= \left\{\mathcal L_1^{n_1} \otimes \cdots \otimes \mathcal L_\lambda^{n_\lambda} \middle| \textup{$(n_1,\ldots,n_\lambda)\in(\mathbb Z_{\ge0})^{\times\lambda}$ with $\sum_{i=1}^\lambda n_i =n$ } \right\} 
\subseteq \mathrm{Pic}^0(Y) 
\end{align*}
for each $n\in\mathbb Z_{>0}$. 
Note that $M(n)\subseteq M(n+1)$, as 
$\mathcal O_Y\in\mathbb L(\mathcal G_0)\subseteq M(1)$. 
Fix an $e\gg0$ with $lm|(p^e-1)$ such that $(p^e-1)(K_X+\Delta)$ is Cartier. 
Put $\nu:=p^e+\frac{1-p^e}{lm}\in\mathbb Z_{>0}$. 
Then $\nu<p^e$. 
Take an $\mathcal L\in M(1)$. Then $\mathcal L\in\mathbb L(\mathcal G_r)$ 
for some $r\in I$. 
Let $q$ and $r'$ be the quotient and the remainder of the division of 
$(r+1)p^e-1$ by $i$. 
Then $r'\in I$ and 
\begin{align*}
&(1-p^e)(K_X+\Delta')+mp^eL-p^er(K_X+\Delta)
\\ & = (1-p^e)(K_X+\Delta)+\left(\frac{1-p^e}{l}+mp^e\right)L
-p^er(K_X+\Delta)
\\ & = (1-(r+1)p^e)(K_X+\Delta)+\nu mL
\\ & = (-r'-qi)(K_X+\Delta)+\nu mL
\\ & \sim_{\mathbb Z} -r'(K_X+\Delta)+iqD+iqf^*P +\nu mL, 
\end{align*}
so we see from Lemma~\ref{lem:semi-ample decomp} that 
there is a surjective morphism 
\begin{align*}
& \bigoplus f_*\mathcal O_X(-r'(K_X+\Delta)+iqf^*P+\nu mL)
\\ & \twoheadrightarrow 
f_*\mathcal O_X((1-p^e)(K_X+\Delta')+mp^e L -p^er(K_X+\Delta))
\end{align*}
Using ($\ast$) for $N=-r(K_X+\Delta)+iqf^*P$, 
we get a generically surjective morphism 
$$
{F_Y^e}_*f_*\mathcal O_X(-r'(K_X+\Delta)+iqf^*P+\nu mL) 
\to \mathcal L. 
$$
By the Grothendieck duality (note that $\omega_Y\cong\mathcal O_Y$), 
we obtain the non-zero (so generically surjective) morphism 
$$
f_*\mathcal O_X(-r'(K_X+\Delta)+iqf^*P+\nu mL) 
\to \mathcal L^{p^e}. 
$$
Considering the generically surjective morphism 
\begin{align*}
\left(\bigotimes^{\nu-1} \mathcal G_0 \right) 
\otimes \mathcal G_{r'}(iqf^*P)
&\cong 
\left( \bigotimes^{\nu-1} f_*\mathcal O_X(mL) \right)
\otimes f_*\mathcal O_X(mL-r'(K_X+\Delta)+iqf^*P)
\\ & \twoheadrightarrow
f_*\mathcal O_X(-r'(K_X+\Delta)+iqf^*P+\nu mL)
\to \mathcal L^{p^e}, 
\end{align*}
we see that 
$\mathcal L^{p^e}\otimes \mathcal M^{-1}(-iqP) \in \mathbb L(\mathcal G_{r'})$
for some $\mathcal M\in M(\nu-1)$, which implies that 
$\mathcal L^{p^e}(-iqP)\in M(\nu)$. 
Note that $\mathcal G_0$ has a filtration whose each quotient 
is isomorphic to an element of $\mathbb L(\mathcal G_0)$. 
Therefore, by the same argument as that of the proof of 
Theorem~\ref{thm:decomp3}, we see that $P\sim_{\mathbb Q}0$ and 
all the elements of $\mathbb L(\mathcal G_0)$ are torsion line bundles. 
This implies that 
$\mathcal G_0 \cong \bigoplus \mathcal U_i \otimes \mathcal L_i$, 
where each $\mathcal U_i$ is a unipotent vector bundle 
and each $\mathcal L_i$ is a torsion line bundle. 

Next, we show that there is an isogeny $\pi:Z\to Y$ such that 
$
\pi^*f_*\mathcal O_X 
\cong \mathcal O_Z^{\oplus \mathrm{rank}\,f_*\mathcal O_X}. 
$ 
Let $f:X\xrightarrow{g} C\xrightarrow{h} Y$ be 
the Stein factorization of $f$. 
Since $C$ is of genus at least one, we see from \cite[Theorem~5.3]{EP23} 
that $-K_C$ is pseudo-effective, so $C$ is of genus one. 
Thus, $\deg h_*\mathcal O_C =0$. 
We show that $h_*\mathcal O_C$ is a direct sum of 
indecomposable vector bundles of degree zero. 
If not, then there is an indecomposable direct summand $\mathcal E$ 
of $h_*\mathcal O_C$ with $\deg \mathcal E>0$. 
Take an $l\in\mathbb Z$ with $l>\frac{\mathrm{rank}\, \mathcal E}{\deg\mathcal E}$. 
Replacing $h:C\to Y$ by its base change along $[l]_Y:Y\to Y$, 
where $[l]_Y$ is a morphism defined by $[l]_Y(y)=l\cdot y$, 
we may assume that $\deg\mathcal E>\mathrm{rank}\,\mathcal E$
(loss of connectivity of $C$). 
Then $\deg \mathcal E(-y)>0$ for a $y\in Y(k)$, so
$$
H^0(C, \mathcal O_C(-h^*y))
\cong H^0(Y,h_*\mathcal O_C(-h^*y)) 
\supseteq H^0(Y, \mathcal E(-y))
\ne0, 
$$ 
a contradiction. Thus, $h_*\mathcal O_C$ is a direct sum of indecomposable 
vector bundles of degree zero. 
By \cite[p.60]{Oda71}, we see that 
$h_*\mathcal O_C\cong \bigoplus_i \mathcal E_i\otimes \mathcal L_i$, 
where $\mathcal E_i$ is an indecomposable vector bundle of degree zero 
with non-zero global sections 
and $\mathcal L_i$ is a line bundle of degree zero. 
Then 
$$
H^0\left(C, h^*\mathcal L_i^{-1}\right)
\cong H^0\left(Y, h_*\mathcal O_C \otimes \mathcal L_i^{-1}\right)
\supseteq H^0(Y, \mathcal E_i) 
\ne 0, 
$$
so $h^*\mathcal L_i\cong \mathcal O_C$, 
which means that each $\mathcal L_i$ is a torsion line bundle. 
Hence, there is an isogeny $\pi:Z\to Y$ such that 
$\pi^*h_*\mathcal O_C \cong \mathcal O_Z^{\oplus \mathrm{rank}\,h_*\mathcal O_C}$. 

Replacing $\pi:Z\to Y$, we may assume that 
$\pi^*f_*\mathcal O_X(mL) \cong \mathcal O_Z^{\oplus \mathrm{rank}\,f_*\mathcal O_X(mL)}$. 
Then the assertion follows 
from an argument similar to that of the proof of Theorem~\ref{thm:decomp1}.
\end{proof}
\section{Decomposition theorems}
\label{section:decomposition}
In this section, we establish decomposition theorems for 
strongly $F$-regular pairs $(X,\Delta)$ such that $-K_X-\Delta$ 
satisfies some positivity condition. 
They are applied to the study of the \'etale fundamental groups of $F$-split 
varieties whose anti-canonical divisors satisfy some positivity conditions. 
The arguments in this section almost the same as that of \cite[\S 11]{PZ19}. 
\subsection{Definitions and lemmas}
In this subsection, we work over a perfect field $k$ of characteristic $p>0$. 
\begin{defn}[\textup{$F$-splitting}] \label{defn:F-split}
Let $(X,\Delta)$ be a normal projective pair 
with $\Delta$ $\mathbb Z_{(p)}$-Weil. 
We say that $(X,\Delta)$ is \textit{$F$-split} if 
$$
\mathcal O_X
\xrightarrow{{F_X^e}^\sharp} {F_X^e}_*\mathcal O_X
\hookrightarrow {F_X^e}_*\mathcal O_X((p^e-1)\Delta)
$$
splits as an $\mathcal O_X$-module homomorphism
for each $e\in\mathbb Z_{>0}$ with $(p^e-1)\Delta$ is integral. 
\end{defn}
\begin{defn}[\textup{Quasi-\'etale morphisms}] \label{defn:quasi-etale}
Let $\pi:X'\to X$ be a morphism between normal varieties. 
We say that $\pi$ is \textit{quasi-\'etale} if 
there exists an open subset $U\subseteq X'$ with 
$\mathrm{codim}(X'\setminus U) \ge 2$
such that $\pi|_U:U\to X$ is \'etale. 
\end{defn}
\begin{defn}[\textup{Augmented irregurality}] \label{defn:regularity}
Let $X$ be a normal projective variety. 
We define the \textit{augmented irregurality $\hat q(X)$} by 
$$
\hat q(X):=\max \{\dim\mathrm{Alb}_X' | \textup{$X'\to X$ is a finite quasi-\'etale morphism of normal varieties}\}
$$
If the maximum does not exists, then we define $\hat q(X)=\infty$. 
\end{defn}
\begin{lem}[\textup{\cite[Lemma~11.1]{PZ19}}] \label{lem:F-split quasi-etale}
Let $(X,\Delta)$ be a normal projective pair with $\Delta$ $\mathbb Z_{(p)}$-Weil. 
Let $\pi:X'\to X$ be a finite quasi-\'etale morphism. 
If $(X,\Delta)$ is strongly $F$-regular $($resp. $F$-split$)$, 
then so is $(X',\pi^*\Delta)$. 
\end{lem}
By this lemma and \cite[Theorems~1.2 and~1.3]{Eji19w}, we see that 
$\hat q(X)\le \dim X$ for a $F$-split projective pair $(X,\Delta)$. 
\begin{prop}[\textup{\cite[Proposition~11.3]{PZ19}}] \label{prop:hat q}
Let $\pi:X'\to X$ be a finite morphism between normal projective varieties. 
If $\pi$ is quasi-\'etale or universally homeomorphic, 
then $\hat q(X')=\hat q(X)$. 
\end{prop}
\subsection{Decomposition theorems}
In this subsection, we work over a perfect field $k$ of characteristic $p>0$. 
\begin{thm}[\textup{cf. \cite[Theorem~11.6]{PZ19}}] \label{thm:F-split semi-ample}
Let $(X,\Delta)$ be a strongly $F$-regular $F$-split pair 
such that $-K_X-\Delta$ is $\mathbb Z_{(p)}$-Cartier 
and is numerically equivalent to 
a semi-ample $\mathbb Q$-Cartier divisor. 
Then there exist finite surjective morphisms 
$$
Y \xrightarrow{\tau} W \xrightarrow{\sigma} X
$$
with the following properties: 
\begin{enumerate}[$(1)$]
\item $\sigma:W\to X$ is quasi-\'etale. 
\item $\tau:Y\to W$ is an infinitesimal torsor under 
$\Pi_{i=1}^{\hat q(X)} \mu_{p^{j_i}}$
for some integers $j_i\ge0$. 
\item $K_Y+\Delta_Y \sim_{\mathbb Z_{(p)}} \tau^*\sigma^*(K_X+\Delta)$, 
where $\Delta_Y:=\tau^*\sigma^*\Delta$. 
\item $(Y, \Delta_Y) \cong (F,\Delta_F)\times_{k'} B$, where 
\begin{enumerate}[$(i)$]
\item $k'$ is a finite field extension of $k$, 
\item $B$ is an ordinary abelian variety over $k'$ of dimension $\hat q(X)$, and 
\item $(F,\Delta_F)$ is a strongly $F$-regular $F$-split geometrically integral projective pair over $k'$ such that $-K_F-\Delta_F$ is a semi-ample $\mathbb Z_{(p)}$-Cartier divisor and $\hat q(F)=0$. 
\end{enumerate}
\item If $X$ is smooth, then $Y$ is Gorenstein. 
\end{enumerate}
\end{thm}
\begin{proof}
By Lemma~\ref{lem:F-split quasi-etale}, we may assume that 
$\dim \mathrm{Alb}_X=\hat q(X)$.
Let $x\in X$ be a closed point. 
Replacing $k$ and $X$ by $\kappa(x)$ and a component of $X_{\kappa(x)}$, respectively, 
we may assume that $X$ has a $k$-rational point. 
Then $A:=\mathrm{Alb}_X$ is an abelian variety 
(cf. \cite[10.3.9]{KLOS21}). 
By assumption, there is a semi-ample $\mathbb Q$-Cartier divisor $D$ on 
$X$ and a numerically trivial $\mathbb Q$-divisor $P$ on $A$ 
such that $-K_X-\Delta\sim_{\mathbb Q}D+f^*P$. 
This follows from the fact that if $L$ is a numerically trivial 
Cartier divisor on $X$, then $mL$ is algebraically trivial for some $m>0$
by \cite[COROLLARY~6.13]{Kle05} 
(this reference claims over an algebraically closed field, 
but one can easily check that it also holds over a perfect field).
Since $(X,\Delta)$ is normal and $F$-split, so is 
$
\left(X_{\overline\eta},\Delta|_{X_{\overline\eta}}\right), 
$
where $X_{\overline\eta}$ is the geometric generic fiber of 
the Albanese morphism $a:X\to A$, 
so we see from \cite[Theorem~6.3]{PZ19} that 
$
\left(X_{\overline\eta},\Delta|_{X_{\overline\eta}}\right) 
$
is strongly $F$-regular. 
Therefore, we can apply Theorem~\ref{thm:decomp3} and 
obtain that $P\sim_{\mathbb Q}0$ and that there is 
an isogeny $\pi:B\to A$ such that 
$$
(X\times_A B, \mathrm{pr}_1^*\Delta) 
\cong (F, \Delta_F)\times_k B 
$$
as $B$-schemes, where $F:=X_0$ and $\Delta_F:=\Delta|_{X_0}$. 
Note that $-K_F-\Delta_F=(-K_X-\Delta)|_F$ is semi-ample. 
Note also that $F$ is geometrically integral, as $X_{\overline\eta}$ 
is integral by \cite[Theorems~1.2 and~1.3]{Eji19w} 
and so we can apply \cite[Th\'eor\`eme~(12.2.4)]{Gro65}. 
Put $n:=\deg \pi$. 
Then by \cite[(5.12)~Proposition]{EVM12}, we see that 
$[n]_A:A\to A$ factors through $\pi$. 
Here, $[n]_A:A\to A$ is the morphism defined by $[n]_A(a)=n\cdot a$. 
We may assume that $B=A$ and $\pi=[n]_A$. 
Write $n=mp^e$ with $p\nmid m$. 
Then $[n]_A=[m]_A\circ [p^e]_A$ and 
$[p^e]_A=V_{A/k}^{(e)} \circ F_{X/k}^{(e)}$, 
where $V_{A/k}^{(e)}$ is the $e$-th Verschiebung homomorphism
(see \cite[(5.21)]{EVM12}). 
Set 
$$
\sigma':=[m]_A\circ V_{A/k}^{(e)}:A':=A_{(\mathrm{Spec}\,k)^e} \to A
\quad \textup{and} \quad 
\tau':=F_{X/k}^{(e)}:B=A\to A'. 
$$
Put $W:=X\times_A A'$ and $Y:=X\times_A B \cong W\times_{A'} B$. 
Let $\sigma:W\to X$ and $\tau:Y\to W$ be the induced morphisms: 
$$
\xymatrix{
Y \ar[d] \ar[r]^\tau \ar@{}[dr]|\square & W \ar[d] \ar[r]^\sigma \ar@{}[dr]|\square & X \ar[d]^a \\
B \ar[r]_{\tau'} & A' \ar[r]_{\sigma'} & A. 
}
$$
Since $(X,\Delta)$ is $F$-split, $A$ is ordinary, so $V_{A/k}^{(e)}$ is \'etale, 
and hence $\sigma'$ is \'etale. Thus,~(1) holds. 
Statement~(2) follows from the fact that $\tau'=F_{X/k}^{(e)}:A\to A'$ 
is a torsor under $\Pi_{i=1}^{\hat q(S)}\mu^{j_i}$ 
for some integers $j_i\ge0$ (\cite[p. 146]{Mum70}). 
We obtain~(3) by 
$$
K_Y+\Delta_Y
\sim_{\mathbb Z_{(p)}} K_{Y/B}+\Delta_Y
\sim_{\mathbb Z_{(p)}} \tau^*\sigma^*(K_{X/A}+\Delta)
\sim_{\mathbb Z_{(p)}} \tau^*\sigma^*(K_X+\Delta). 
$$
To show~(4), we have to check that 
\begin{enumerate}[(a)]
\item $(F,\Delta_F)=(X_0,\Delta|_{X_0})$ is strongly $F$-regular,
\item $(F,\Delta_F)=(X_0,\Delta|_{X_0})$ is $F$-split and 
\item $\hat q(F)=0$. 
\end{enumerate}
Statement~(a) is obtained by the strong $F$-regularity of general fibers of $a:X\to A$ (\cite[Theorem~B]{PSZ18}). 
Statement~(b) follows from the $F$-splitting of $(X,\Delta)$ 
(cf. \cite{Eji19w}). 
We show~(c). Suppose that $\hat q(F)>0$. 
Then $\hat q(Y)> \dim B=\hat q(X)$, 
which contradicts Proposition~\ref{prop:hat q}, 
as $\tau:Y\to W$ is a universal homeomorphism. 
Statement~(5) follows from the fact that $a:X\to A$ is a Gorenstein morphism. 
\end{proof}
\begin{rem} \label{rem:char zero}
In characteristic zero, a statement similar to 
Theorem~\ref{thm:F-split semi-ample}
follows from Ambro's theorem \cite[Theorem~0.1]{Amb05}. 
Let $k$ be an algebraically closed field of characteristic zero. 
Let $(X,\Delta)$ be a klt projective pair over $k$ such that 
$-K_X-\Delta$ is numerically equivalent to a semi-ample $\mathbb Q$-Cartier 
divisor $D$ on $X$. 
Replacing $X$ with a finite quasi-\'etale cover, we may assume that 
$\dim \mathrm{Alb}_X =\hat q(X)$. 
By the Bertini theorem, we can find an effective $\mathbb Q$-Cartier divisor 
$E$ on $X$ such that $E\sim_{\mathbb Q}D$ and 
$(X, \Delta':=\Delta+E)$ is klt. 
Then $K_X+\Delta'\equiv0$. 
Hence, by \cite[Theorem~0.1]{Amb05}, there is a finite \'etale cover 
$B\to A:=\mathrm{Alb}_X$ such that 
$
(X,\Delta')\times_A B \cong (F,\Delta'|_F)\times_k B
$
as $B$-schemes, where $F$ is a closed fiber of the Albanese morphism of $X$. 
We see that $\hat q(F)=0$ and that 
$
(X,\Delta)\times_A B \cong (F,\Delta|_F)\times_k B. 
$
Note that the Bertini theorem does not hold in positive characteristic, 
so we cannot reduce the proof of Theorem~\ref{thm:F-split semi-ample} 
to the case of $K_X+\Delta\equiv 0$. 
\end{rem}
By the same argument as the proof of Theorem~\ref{thm:F-split semi-ample}, 
we can prove the following theorem: 
\begin{thm} \label{thm:F-split nef}
Suppose that $k\subseteq \overline{\mathbb F_p}$. 
Let $(X,\Delta)$ be a strongly $F$-regular $F$-split projective pair 
such that $-K_X-\Delta$ is a nef $\mathbb Z_{(p)}$-Cartier divisor. 
Then there exist finite surjective morphisms 
$$
Y\xrightarrow{\tau}W\xrightarrow{\sigma}X
$$
with the following properties: 
\begin{enumerate}[$(1)$]
\item $\sigma:W\to X$ is quasi-\'etale. 
\item $\tau:Y\to W$ is an infinitesimal torsor under 
$\Pi_{i=1}^{\hat q(X)} \mu_{p^{j_i}}$
for some integers $j_i\ge0$. 
\item $K_Y+\Delta_Y \sim_{\mathbb Z_{(p)}} \tau^*\sigma^*(K_X+\Delta)$, 
	where $\Delta_Y:=\tau^*\sigma^*\Delta$. 
\item $(Y, \Delta_Y) \cong (F,\Delta_F)\times_{k'} B$, where 
\begin{enumerate}[$(i)$]
\item $k'$ is a finite field extension of $k$, 
\item $B$ is an ordinary abelian variety over $k'$ of dimension $\hat q(X)$, and 
\item $(F,\Delta_F)$ is a strongly $F$-regular $F$-split projective pair
	over $k'$ such that $-K_F-\Delta_F$ is a nef $\mathbb Z_{(p)}$-Cartier divisor and $\hat q(F)=0$. 
\end{enumerate}
\item If $X$ is smooth, then $Y$ is Gorenstein. 
\end{enumerate}
\end{thm}
\begin{proof}
This follows from the same argument as that of the proof of Theorem~\ref{thm:F-split semi-ample} and Theorem~\ref{thm:decomp1}. 
Note that we use the fact that an \'etale cover of an abelian variety is 
again abelian variety~\cite[p.167, Thoerem]{Mum70}. 
\end{proof}
Next, we treat varieties with nef tangent bundle. 
\begin{thm} \label{thm:nef tangent}
Let $k$ be an algebraically closed field of characteristic $p>0$. 
Let $X$ be a smooth $F$-split projective variety over $k$ 
with nef tangent bundle. 
Suppose that either $k=\overline{\mathbb F_p}$ or $-K_X$ is numerically equivalent to a semi-ample $\mathbb Q$-divisor on $X$. 
Then there exist finite surjective morphisms 
$$
F\times_k B\xrightarrow{\tau}W\xrightarrow{\sigma}X
$$
with the following properties: 
\begin{enumerate}[$(1)$]
\item $\sigma:W\to X$ is \'etale. 
\item $\tau:Y\to W$ is an infinitesimal torsor under 
$\pi^{\hat q(X)}_{i=1} \mu_{p^{j_i}}$ for some integers $j_i\ge 0$. 
\item $K_{F\times_k B}\sim \tau^*\sigma^*K_X$. 
\item $B$ is an ordinary abelian variety of dimension $\hat q(X)$. 
\item $F$ is a smooth $F$-split separably rationally connected Fano variety 
with nef tangent bundle. 
\end{enumerate}
In particular, $-K_X$ is semi-ample. 
\end{thm}
\begin{proof}
By \cite[Theorem~1.3]{EY23} (cf. \cite[Theorem~1.8]{KW23}), we may assume that 
$\dim \mathrm{Alb}_X=\hat q(X)$ and 
a closed fiber $F$ of the Albanese morphism of $X$ is 
a smooth $F$-split separably rationally connected Fano variety with 
nef tangent bundle. 
Then the remaining of the proof is the same as that of 
Theorem~\ref{thm:F-split semi-ample}. 
\end{proof}
At the end of this subsection, 
we consider the case when the pair is not necessarily $F$-split.
\begin{thm} \label{thm:hat q=1}
Let $(X,\Delta)$ be a strongly $F$-regular projective pair 
$($not necessarily $F$-split$)$
such that $-K_X-\Delta$ is $\mathbb Z_{(p)}$-Cartier and is 
numerically equivalent to a semi-ample $\mathbb Q$-Cartier divisor. 
Suppose that $\hat q(X)=1$. 
Then there exist finite surjective morphisms 
$$
Y\xrightarrow{\tau}W\xrightarrow{\sigma}X
$$
with the following properties: 
\begin{enumerate}[$(1)$]
\item $\sigma:W\to X$ is quasi-\'etale. 
\item $\tau:Y\to W$ is an infinitesimal torsor. 
\item $K_Y+\Delta_Y\sim_{\mathbb Z_{(p)}}\tau^*\sigma^*(K_X+\Delta)$, 
	where $\Delta_Y:=\tau^*\sigma^*\Delta$. 
\item $(Y,\Delta_Y)\cong (F,\Delta_F)\times_{k'} E$, where
\begin{enumerate}[$(i)$]
\item $k'$ is a finite field extension of $k$, 
\item $E$ is an elliptic curve over $k'$, and 
\item $(F,\Delta_F)$ is geometrically integral projective pair over $k'$ 
	such that $-K_F-\Delta_F$ is a semi-ample $\mathbb Z_{(p)}$-Cartier
	divisor and $\hat q(F)=0$. 
\end{enumerate}
\item If $X$ is smooth, then $Y$ is Gorenstein. 
\end{enumerate}
\end{thm}
\begin{proof}
By the same argument as that of the proof of 
Theorem~\ref{thm:F-split semi-ample}, 
we may assume that $A:=\mathrm{Alb}_X$ is an elliptic curve. 
By assumption, there is a semi-ample $\mathbb Q$-Cartier divisor $D$ on $X$ 
and a numerically trivial $\mathbb Q$-divisor $P$ on $A$
such that $-K_X-\Delta\sim_{\mathbb Q}D+f^*P$. 
Therefore, we can apply Theorem~\ref{thm:decomp ell} and obtain that 
$P\sim_{\mathbb Q}0$ and that there is an isogeny $\pi:E\to A$ such that 
$$
(X\times_A E, \mathrm{pr}_1^*\Delta) \cong (F,\Delta_F)\times_k E
$$
as $E$-schemes, where $F:=X_0$ and $\Delta_F:=\Delta|_{X_0}$. 
Note that $-K_F-\Delta_F=(-K_X-\Delta)|_F$ is semi-ample. 
Note also that $F$ is geometrically integral, as the geometric generic fiber
of the Albanese morphism is integral by \cite[Theorem~1.3~(d)]{EP23}
and so we can apply \cite[Th\'eor\`eme~(12.2.4)]{Gro65}.
Let $B$ be the normalization of $A$ in the separable closure of $K(A)$ 
in $K(E)$. 
Then the induced morphism $\sigma':B\to A$ is \'etale and 
$\tau':E\to B$ is an iterated Frobenius morphism. 
Put $W:=X\times_A B$ and $Y:=X\times_A E\cong W\times_B E$. 
Let $\sigma:W\to X$ and $\tau:Y\to W$ be the induced morphism: 
$$
\xymatrix{
Y \ar[d] \ar[r]^\tau \ar@{}[dr]|\square & W \ar[d] \ar[r]^\sigma \ar@{}[dr]|\square & X \ar[d]^a \\
E \ar[r]_{\tau'} & B \ar[r]_{\sigma'} & A. 
}
$$
Since $\sigma'$ is \'etale, so is $\sigma$. 
Also, since $\tau'$ is an infinitesimal torsor, so is $\tau$. 
The remaining statements follows from an argument similar to 
that of the proof of Theorem~\ref{thm:F-split semi-ample}. 
\end{proof}
\subsection{Applications}
We first enumerate corollaries directly follows from 
Theorems~\ref{thm:F-split semi-ample} and~\ref{thm:hat q=1}. 
\begin{cor}[\textup{of Theorem~\ref{thm:F-split semi-ample}}] \label{cor:F-split semi-ample}
Let $(X,\Delta)$ be a strongly $F$-regular $F$-split projective pair over 
a perfect field of characteristic $p>0$ such that 
$-K_X-\Delta$ is $\mathbb Z_{(p)}$-Cartier and is 
numerically equivalent to a semi-ample $\mathbb Q$-Cartier divisor. 
Then $-K_X-\Delta$ is semi-ample. 
\end{cor}
\begin{cor}[\textup{of Theorem~\ref{thm:hat q=1}}] \label{cor:hat q=1} 
Let $(X,\Delta)$ be a strongly $F$-regular projective pair 
$($not necessarily $F$-split$)$ over 
a perfect field of characteristic $p>0$ such that 
$-K_X-\Delta$ is $\mathbb Z_{(p)}$-Cartier and is 
numerically equivalent to a semi-ample $\mathbb Q$-Cartier divisor. 
Suppose that $\hat q(X)\le 1$. 
Then $-K_X-\Delta$ is semi-ample. 
In particular, if $\hat q(X)\le 1$ and $K_X+\Delta\equiv 0$, 
then $K_X+\Delta\sim_{\mathbb Q}0$. 
\end{cor}
\begin{proof}
The case of $\hat q(X)=1$ follows from Theorem~\ref{thm:hat q=1}. 
The case of $\hat q(X)=0$ follows from the fact that, 
if $\mathrm{Alb}_X$ is a point, then the numerical equivalence is 
equivalent to the $\mathbb Q$-linear equivalence. 
\end{proof}
\begin{cor}[\textup{of Theorem~\ref{thm:hat q=1}}] \label{cor:surface}
Let $(X,\Delta)$ be a strongly $F$-regular projective pair 
$($not necessarily $F$-split$)$ over 
a perfect field of characteristic $p>0$ such that 
$-K_X-\Delta$ is $\mathbb Z_{(p)}$-Cartier and 
is numerically equivalent to a semi-ample $\mathbb Q$-Cartier divisor. 
Suppose that $\dim X=2$. 
Then $-K_X-\Delta$ is semi-ample. 
\end{cor}
\begin{proof}
By \cite[Theorem~1.3]{EP23}, we see that $\hat q(X)=0,1$ or $2$. 
The case $\hat q(X)=1$ follows from Corollary~\ref{cor:hat q=1}. 
Let $a:X\to A$ be the Albanese morphism of $X$. 
We may assume that $\dim A=\hat q(X)$. 
When $\dim A=0$, the numerical equivalence is equivalent to 
the $\mathbb Q$-linear equivalence, so the assertion is obvious. 
When $\dim A=2$, $X$ is an abelian variety by \cite[Proposition~3.2]{EP23}, 
so $K_X\sim0$ and $\Delta=0$. 
\end{proof}
Next, we investigate the \'etale fundamental groups 
of $F$-split varieties whose anti-canonical divisors 
satisfy some positivity conditions. 
The argument in this subsection is also almost the same as that of 
\cite[\S 12]{PZ19}.
\begin{lem}[\textup{\cite[Lemma~4.4.17]{Del73}}] \label{lem:pi}
Let $f:X\to Y$ be a generically finite surjective morphism 
between normal projective varieties over an algebraically closed field. 
Then the image of the induced morphism 
$f_*:\pi^{\textup{\'et}}(X)\to\pi^{\textup{\'et}}(Y)$ 
is of finite index. 
\end{lem}
\begin{cor}[\textup{of Theorems~\ref{thm:F-split semi-ample} and~\ref{thm:F-split nef}}] \label{cor:pi}
Let $X$ be a smooth $F$-split projective variety 
over an algebraically closed field $k$ of characteristic $p>0$. 
Suppose that either
\begin{itemize}
\item $-K_X$ is semi-ample, or
\item $-K_X$ is nef and $k=\overline{\mathbb F_p}$. 
\end{itemize}
If $\hat q(X) \ge \dim X-2$ $($e.g., $\dim X=3$ and $\hat q(X)>0$$)$, then $\pi^{\textup{\'et}}(X)$ is virtually abelian. I.e., there exists an abelian subgroup of finite index.  
\end{cor}
\begin{proof}
If $K_X\sim_{\mathbb Q}0$, then	the statement follows from \cite[Corollary~12.5]{PZ19}. 
We assume that $\kappa(X)=-\infty$. 
Let $Y\cong F\times_k B \xrightarrow{\tau}W\xrightarrow{\sigma} X$ 
be as in Theorems~\ref{thm:F-split semi-ample} or~\ref{thm:F-split nef}. 
By Lemma~\ref{lem:pi}, it is enough to show that $\pi^{\textup{\'et}}(F)$
is finite. 
Note that $\kappa(K_F)=-\infty$, as $\kappa(X)=-\infty$. 
If $\hat q(X)=\dim X$, then $F$ is a point. 
If $\hat q(X)=\dim X-1$, then $F\cong\mathbb P^1$, 
so $\pi^{\textup{\'et}}(X)=0$. 
If $\hat q(X)=\dim X-2$, then $F$ is a Gorenstein strongly $F$-regular projective surface with $-K_F$ semi-ample, and so $F$ has canonical singularities. 
Taking the minimal resolution, we may assume that $F$ is smooth. 
Let $F'$ be a minimal surface (surface without $(-1)$-curves) of $F$. 
Since $\kappa(F)=-\infty$, we see that $F'\cong\mathbb P^2$ or 
a ruled surface over $C$. 
We consider the case of ruled surface. 
As $\hat q(F')=0$, the base curve $C$ is $\mathbb P^1$. 
Thus, $F'$ is rational, and hence $\pi^{\textup{\'et}}(F)=0$. 
\end{proof}
\begin{cor}[\textup{of Theorem~\ref{thm:nef tangent}}] 
\label{cor:pi nef tangent}
Let $k$ be an algebraically closed field of characteristic $p>0$. 
Let $X$ be a smooth $F$-split projective variety over $k$ with 
nef tangent bundle $T_X$. 
Suppose that either $-K_X$ is semi-ample or $k=\overline{\mathbb F_p}$. 
Then $\pi^{\textup{\'et}}(X)$ is virtually abelian. 
I.e., there exists an abelian subgroup of finite index. 
\end{cor}
\begin{proof}
Let $F\times_k B\xrightarrow{\tau}W\xrightarrow{\sigma}X$ be 
as in Theorem~\ref{thm:nef tangent}. 
Then $\pi^{\textup{\'et}}(F)=0$ by \cite[Corollary~1.4]{KW23}. 
Thus, the assertion follows from Lemma~\ref{lem:pi}. 
\end{proof}
In the remaining of this subsection, 
we study rational points of $F$-split varieties over finite fields
with nef anti-canonical divisors. 
\begin{cor}[\textup{of Theorem~\ref{thm:decomp1}}] \label{cor:rational points}
Let $(X,\Delta)$ be a strongly $F$-regular $F$-split geometrically integral 
projective pair over $\mathbb F_q$ for some $q=p^e>19$. 
Suppose that either 
\begin{enumerate}[$(a)$]
\item $-K_X-\Delta$ is a nef $\mathbb Z_{(p)}$-divisor,  
$K_X\not\sim_{\mathbb Q}0$, 
$\frac{1}{2}b_1(X) \ge \dim X -3$ $($e.g., $\dim X=4$ and $b_1(X)>0$$)$ 
and $p>5$, or  
\item $\Delta=0$, $K_X$ is a numerically trivial Cartier divisor,  
and $\frac{1}{2}b_1(X) \ge \dim X-2$ 
$($e.g., $\dim X=3$ and $b_1(X)>0$$)$. 
\end{enumerate}
Then $X(\mathbb F_q)\ne \emptyset$. 
\end{cor}
\begin{proof}
Let $a:X\to A$ be the Albanese morphism of $X$. 
Then $a$ is separable and has connected fibers by \cite[Theorems~1.2 and~1.3]{Eji19w}, 
so the geometric generic fiber $X_{\overline\eta}$ is integral. 
Since $A_{\overline k}$ is an abelian variety, by \cite[Theorem~3]{Lan55}, 
we see that there is an $y\in A(\mathbb F_q)$.
Let $F$ be the fiber of $a$ over $y$. 
Since $(X,\Delta)$ is $F$-split, 
$\left(X_{\overline\eta},\Delta|_{X_{\overline\eta}}\right)$
is strongly $F$-regular by \cite[Theorem~6.3]{PZ19}, 
so we can apply Theorem~\ref{thm:decomp1}. 
Thus, $F$ is geometrically integral by \cite[Th\'eor\`eme~(12.2.4)]{Gro65}.
We first deal with case~$(a)$. 
Since $\dim F=\dim X- \frac{1}{2}b_1(X) \le 3$, we can take a resolution 
$\pi:\tilde F\to F$ of $F$. 
We show that $\kappa\left(\tilde F\right)=-\infty$. 
If $\kappa\left(\tilde F\right)\ge0$, 
then there is an effective $\mathbb Q$-divisor 
$\tilde E$ on $\tilde F$ with $\tilde E\sim_{\mathbb Q} K_{\tilde F}$. 
Then $0 \le E:=\pi_*\tilde E\sim_{\mathbb Q} K_F$, 
so $-E-\Delta_F$ is nef, and hence $E=\Delta_F=0$ and $K_F\sim_{\mathbb Q}0$.
This means that $\Delta=0$ and $K_X\sim_{\mathbb Q}0$ by Theorem~\ref{thm:decomp1}, 
but this contradicts the assumption. 
Thus $\kappa\left(\tilde F\right)=-\infty$. 
Then we see from \cite[Theorem~1.3]{BF23} that $F(\mathbb F_q)\ne\emptyset$, 
so $X(\mathbb F_q)\ne\emptyset$. 
Next, we consider case~$(b)$. 
If $\dim F=1$, then $\hat q(F)=0$ and $K_F\equiv 0$, a contradiction. 
When $\dim F=2$, $F$ is a Gorenstein surface 
with strongly $F$-regular singularities such that 
$K_F$ is a numerically trivial divisor, 
so $F$ is canonical, and hence it follows from \cite[Proposition~3.4]{BF23} 
that $F(\mathbb F_q)\ne\emptyset$. 
Thus $X(\mathbb F_q)\ne\emptyset$. 
\end{proof}
\section{Algebraic fiber spaces over the projective line}
\label{section:P^1}
In this section, we show that Theorem~\ref{thm:decomp1-intro} (Theorem~\ref{thm:decomp1}) holds over the projective line without the assumption that the geometric generic fiber is strongly $F$-regular.  

We work over a perfect field $k$ of characteristic $p>0$. 
\begin{thm} \label{thm:decomp P^1}
Let $(X,\Delta)$ be a strongly $F$-regular projective pair
and let $Y$ be the projective line. 
Let $f:X\to Y$ be a surjective morphism. 
Suppose that 
\begin{itemize}
\item $K_X+\Delta$ is $\mathbb Z_{(p)}$-Cartier, and 
\item $-K_{X/Y}-\Delta$ is nef. 
\end{itemize}
Then $(X,\Delta) \cong (X_y,\Delta|_{X_y}) \times_k Y$ 
as $Y$-schemes, where $X_y$ is the fiber of $f$ over a $k$-rational point $y\in Y(k)$. 
\end{thm}
We prove this theorem by mimicking an argument in \cite[\S 7]{EP23} that 
imitates an argument in \cite[\S 5]{PZ19}. 
\begin{thm} \label{thm:semi-ample}
Let $(X,\Delta)$ be a $F$-pure projective pair
and let $Y$ be a projective variety. 
Let $f:X\to Y$ be a morphism. 
Let $L$ be a nef $\mathbb Q$-Cartier divisor on $X$. 
Suppose that 
\begin{itemize}
\item $K_X+\Delta$ is $\mathbb Z_{(p)}$-Cartier, and 
\item $K_X+\Delta+L$ is $f$-ample. 
\end{itemize}
Let $H$ be an ample Cartier divisor on $Y$ and let $j$ be the smallest 
positive integer such that $|jH|$ is free. 
Then $K_X+\Delta+L+(jn+1)f^*H$ is semi-ample, where $n:=\dim Y$. 
\end{thm}
\begin{proof}
Let $m$ be an integer large and divisible enough. 
Then by Theorem~\ref{thm:PS-type}, we see that 
$$
f_*\mathcal O_X(m(K_X+\Delta+L)) (m(jn+1)H)
\cong 
f_*\mathcal O_X(m(K_X+\Delta+L+(jn+1)f^*H))
$$
is globally generated. By the $f$-ampleness, the natural morphism 
$$
f^* f_*\mathcal O_X(m(K_X+\Delta+L+(jn+1)f^*H))
\twoheadrightarrow 
\mathcal O_X(m(K_X+\Delta+L+(jn+1)f^*H))
$$
is surjective. Thus, $|m(K_X+\Delta+L+(jn+1)f^*H)|$
is free. 
\end{proof}
Note that if $Y\cong\mathbb P^1$ and $H \in|\mathcal O_Y(1)|$, 
then we see that $K_{X/Y}+\Delta+L$ is semi-ample. 
\begin{proof}[Proof of Theorem~\ref{thm:decomp P^1}]
\setcounter{step}{0}
\begin{step} \label{step:eff nef}
In this step, we show that every $f$-ample effective 
$\mathbb Q$-Cartier divisor on $X$ is semi-ample. 
Let $\Gamma$ be an effective $\mathbb Q$-Cartier divisor on $X$. 
Then there is an $\varepsilon\in\mathbb Q_{>0}$ such that 
$\varepsilon \Gamma$ is $\mathbb Z_{(p)}$-Cartier 
and $(X,\Delta+\varepsilon\Gamma)$ is $F$-pure. 
Then 
$$
\underbrace{\varepsilon\Gamma}_{\textup{$f$-ample}}
\sim K_{X/Y}+\Delta + \varepsilon\Gamma \underbrace{-K_{X/Y}-\Delta}_{\textup{nef}}, 
$$
so by Theorem~\ref{thm:semi-ample}, we see that $\Gamma$ is semi-ample. 
\end{step}
Put $d:=\dim X-1$. 
Let $\tilde L$ be a very ample Cartier divisor on $X$ such that 
$R^if_*\mathcal O_X(\tilde L)=0$ for each $i>0$. 
Let $y\in Y$ be a closed point. 
Set $L:=\mu\tilde L -\nu f^*y$, where $\mu:=(d+1)\tilde L^d\cdot f^*y$ 
and $\nu:=\tilde L^{d+1}$. 
Then $L$ is $f$-ample and $L^{d+1}=0$. 
\begin{step} \label{step:L nef}
In this step, we prove that $L$ is nef. 
Let $A$ be an ample $\mathbb Z_{(p)}$-divisor on $Y$. 
It is enough to show that $L+f^*A$ is nef. 
By an argument in \cite[Proof of Theorem~5.4]{PZ19}, 
we see that $\Gamma\sim L+f^*A$ for some effective $\mathbb Z_{(p)}$-Cartier divisor $\Gamma$ on $X$.  
Since $L$ is $f$-ample, so is $\Gamma$. 
Hence, $\Gamma$ is nef by Step~\ref{step:eff nef}, 
and so $L+f^*A$ is nef. 
\end{step}
\begin{step} \label{step:anti-nef}
In this step, we show that $f_*\mathcal O_X(mL)$ is anti-nef 
for each $m\in\mathbb Z_{>0}$. 
Fix an $m\in\mathbb Z_{>0}$. 
Since $Y\cong\mathbb P^1$, it follows from Grothendieck's theorem that 
$$
f_*\mathcal O_X(mL)
\cong 
\mathcal O_Y(a_1) \oplus \cdots \oplus \mathcal O_Y(a_r)
$$
for some $a_1,\ldots, a_r\in\mathbb Z$. 
We show that $a_i\le 0$ for each $i$. 
Suppose that $a_i \ge 1$. 
Then for a $k$-rational point $y\in Y$, we have 
$$
0 \ne \mathcal O_Y(a_i)(-y) 
\subseteq H^0(Y,f_*\mathcal O_Y(mL -f^*y))
\cong H^0(X, \mathcal O_Y(mL-f^*y)), 
$$
so there is a member $\Gamma\in|mL-f^*y|$.
Then 
$$
0=L^{d+1}
=L^d(\Gamma +f^*y) 
=\underbrace{L^d\cdot\Gamma}_{\textup{$\ge0$ by Step~\ref{step:L nef}}} + \underbrace{L^d\cdot f^*y}_{\textup{$>0$ by $f$-ampleness of $L$}}
>0, 
$$
a contradiction. Thus, $a_i\le0$ for each $i$. 
\end{step}
\begin{step} \label{step:nef}
In this step, we prove that there is an $m_0\in\mathbb Z_{>0}$ such that 
$f_*\mathcal O_X(mL)$ is nef for each $m\ge m_0$. 
Since $L$ is $f$-ample, there is an $m_0$ such that 
$$
S^0f_*(\sigma(X,\Delta)\otimes \mathcal O_X(mL))
= f_*\mathcal O_X(mL)
$$
for each $m\ge m_0$ (for the definition of $S^0f_*(\sigma(X,\Delta)\otimes \mathcal O_X(mL))$, see \cite[Definition~2.14]{HX15}). 
Fix an $m\ge m_0$. Using Grothendieck's theorem, we obtain the decomposition 
$$
f_*\mathcal O_X(mL) 
\cong 
\mathcal O_Y(a_1) \oplus \cdots \oplus \mathcal O_Y(a_r)
$$
for some $a_1,\ldots,a_r\in\mathbb Z$. 
Take an $l\ge 3$ with $p\nmid l$. 
We define $\pi:Y\to Y$ as the $l$-th toric Frobenius morphism:
$$
\pi: Y \ni (a_0:a_1) \mapsto \left( a_0^l:a_1^l \right) \in Y.
$$
Let $g:Z\to Y$ be the morphism obtained by the base change of 
$f:X\to Y$ along $\pi:Y\to Y$. 
Let $\Delta_Z$ and $L_Z$ be the pullback of $\Delta$ and $L$ to $Z$, 
respectively. 
Put $V:=Y\setminus\{(0:1),(1:0)\}$. 
Then $\pi$ is \'etale over $V$, so one can easily see that 
\begin{align*}
S^0g_*(\sigma(Z,\Delta_Z)\otimes\mathcal O_Z(mL_Z))|_V
& \cong \pi|_V^*S^0f_*(\sigma(X,\Delta)\otimes\mathcal O_X(mL)) |_V
\\ & \cong \pi|_V^*f_*\mathcal O_X(mL)|_V
 \cong g_*\mathcal O_Z(mL_Z)|_V. 
\end{align*}
Thus, the natural inclusion 
$$
S^0g_*(\sigma(Z,\Delta_Z)\otimes\mathcal O_Z(mL_Z))
\hookrightarrow g_*\mathcal O_Z(mL_Z) 
\cong \mathcal O_Y(la_1) \oplus \cdots \oplus \mathcal O_Y(la_r)
$$
is generically isomorphic. 
Since 
\begin{align*}
mL_Z -K_Z-\Delta_Z
& \sim mL_Z -K_{Z/Y}-\Delta_Z -g^*K_Y
\\ & \sim \underbrace{mL_Z}_{\textup{nef and $g$-ample}} \underbrace{-\tau^*(K_{X/Y}+\Delta)}_{\textup{nef}} \underbrace{-g^*K_Y}_{\textup{nef}}
\end{align*}
is nef and $g$-ample, where $\tau:Z\to X$ is the induced morphism, 
we see from \cite[Theorem~1.2]{Eji23} that 
$$
S^0g_*(\sigma(Z,\Delta_Z)\otimes\mathcal O_Z(mL_Z)) \otimes \mathcal O_Y(2)
$$
is globally generated. 
Note that \cite[Theorem~1.2]{Eji23} requires that $Z$ is normal, 
but the same statement can be proved for the case when 
$Z$ satisfies $S_2$ and $G_1$ by the same argument. 
Note also that $Z$ satisfies $S_2$ and $G_1$, 
since $\pi$ is (and so $\tau$ is) a Gorenstein morphism. 
Therefore, we obtain that $la_i+2 \ge 0$, 
so $a_i \ge -\frac{2}{l} \ge -\frac{2}{3}$, 
and hence $a_i \ge 0$ for each $i$. 
\end{step}
Replacing $L$ by $mL$ for $m\ge m_0$, we may assume that 
$f_*\mathcal O_X(mL) \cong \mathcal O_Y^{\oplus r_m}$ for each 
$m\in\mathbb Z_{>0}$, where $r_m:=\mathrm{rank}\,f_*\mathcal O_X(mL)$.  
\begin{step}
In this step, we show that $f_*\mathcal O_X \cong \mathcal O_Y$. 
Let $f:X\xrightarrow{g} Z\xrightarrow{\pi} Y$ be the Stein factorization. 
Taking the base extension, we may assume that $k$ is algebraically closed. 
Let $a_X:X\to A_X$ (resp. $a_Z:Z\to A_Z$) be the Albanese morphism of $A$ 
(resp. $Z$). 
Since $-K_X-\Delta=-K_{X/Y}-\Delta-f^*K_Y$ is nef, 
$a_X$ is surjective by \cite[Theorem~1.3]{EP23}, 
so the induced morphism $A_X\to A_Z$ surjects onto $a_Z(Z)$, 
and hence $Z$ is an elliptic curve or $\mathbb P^1$. 
Suppose that $Z$ is an elliptic curve. 
Let $A$ be an ample Cartier divisor on $X$. 
Since 
$$
-K_X-\Delta +g^*\pi^*K_Y \sim -K_{X/Y}-\Delta
$$
is nef, for each $l\in\mathbb Z_{>0}$, there is an effective 
$\mathbb Z_{(p)}$-Cartier effective divisor $\Gamma$ such that 
$$
-l(K_X+\Delta) +lg^*\pi^*K_Y +A \sim_{\mathbb Z_{(p)}} \Gamma
$$
and that $\left(X, \Delta+\frac{1}{l}\Gamma\right)$ is $F$-pure. 
Then 
$$
K_X+\Delta+\frac{1}{l}\Gamma 
\equiv 
\underbrace{g^*\pi^*K_Y +\frac{1}{l}A}_{\textup{$f$-ample}}, 
$$
so it follows from \cite[Theorem~7.1]{EP23} that $g^*\pi^*K_Y +\frac{1}{l}A$
is nef. Taking $l\to\infty$, we get that $g^*\pi^*K_Y$ is nef, 
so $K_Y$ is also nef, a contradiction. 
Hence, $Z\cong\mathbb P^1$.
Since 
$
\pi_*g_*\mathcal O_X(L)
= f_*\mathcal O_X(L) 
\cong \mathcal O_Y^{\oplus r_1}, 
$
the natural surjective morphism 
$$
\pi^*\pi_*g_*\mathcal O_X(L) 
\twoheadrightarrow g_*\mathcal O_X(L)
$$
implies that $g_*\mathcal O_X(L)$ is globally generated. 
Also, by the same argument as that of Step~\ref{step:anti-nef}, 
we see that $g_*\mathcal O_X(L)$ is anti-nef, 
and hence $g_*\mathcal O_X(L) \cong \mathcal O_Z^{\oplus \rho}$ for some 
$\rho\in\mathbb Z_{>0}$. 
Then 
$$
\rho =h^0(Z,g_*\mathcal O_X(L))
=h^0(Y,\pi_*g_*\mathcal O_X(L))
=\mathrm{rank}\,\pi_*g_*\mathcal O_X(L) 
=\deg\pi \cdot \rho. 
$$
Thus $\pi$ is an isomorphism, and hence $f_*\mathcal O_X \cong \mathcal O_Y$. 
\end{step}
Since $L$ is $f$-ample and 
$f_*\mathcal O_X(mL) \cong \mathcal O_Y^{\oplus r_m}$ 
for each $m\ge 0$, we see that $X\cong Y\times_k X_y$ as $Y$-schemes, 
where $X_y$ is the fiber over a $k$-rational point $y\in Y(k)$. 
It follows from \cite[Lemma~8.4]{PZ19} that $\Delta=\mathrm{pr}_2^*\Delta|_{X_y}$. 
\end{proof}
\section{Examples} \label{section:examples}
In this section, we collect examples which show that assumptions in Theorem~\ref{thm:decomp1-intro} (\ref{thm:decomp1}) and the decomposition theorems cannot be dropped. 
\begin{eg} \label{eg:not F_p}
This example shows that when $k \not\subseteq \overline{\mathbb F_p}$ and $\pi^{\textup{\'et}}(Y)$ is not finite, there exists an algebraic fiber space that does not split after taking the base change along any proper morphism. 

Let $Y$ be an elliptic curve over an algebraically closed field $k\ne \overline{\mathbb F_p}$. Then there exists a numerically trivial line bundle $\mathcal L$ on $Y$ that is not torsion. Put $X:=\mathbb P_Y(\mathcal O_Y\oplus \mathcal L)$ and let $f:X\to Y$ be the natural morphism. Then, although $-K_{X/Y}$ is nef,  $f$ does not splits after taking the base change along any proper morphism $g:Z\to Y$, since $g^*\mathcal L \not\cong\mathcal O_Z$. 
\end{eg}
\begin{eg} \label{eg:indecomp}
In this example, we prove that $X$ in Example~\ref{eg:not F_p} also 
shows that Theorems~\ref{thm:F-split semi-ample},~\ref{thm:F-split nef}, 
\ref{thm:nef tangent} and~\ref{thm:hat q=1}
cannot be generalized to the case when $k\not\subseteq\overline{\mathbb F_p}$ 
and $-K_X$ is nef but not numerically equivalent to a semi-ample
$\mathbb Q$-divisor. 

We may assume that $Y$ is an ordinary elliptic curve. 
Then $X$ is $F$-split by \cite[Proposition~3.1]{GT16}. 
We show that $\hat q(X)=1$. Let $\pi:X'\to X$ be a finite \'etale 
morphism. Then one can easily see that $\pi$ coincides with the base change 
along $f$ of a finite \'etale cover of $Y$. Thus $\dim \mathrm{Alb}_{X'}=1$ 
and $\hat q(X)=1$. 
The nefness of $T_X$ follows from the exact sequence 
$$
0 \to \omega_{X/Y}^{-1} \to T_X \to f^*\omega_Y^{-1} \to 0. 
$$
Assume that $k$ is algebraically closed. 
Suppose that there are finite morphisms 
$$
F\times_k B \xrightarrow{\tau} W \xrightarrow{\sigma} X
$$
that satisfies common conditions in 
Theorems~\ref{thm:F-split semi-ample}, \ref{thm:F-split nef}, 
\ref{thm:nef tangent} and \ref{thm:hat q=1}. 
Then $\kappa(-\sigma^*K_X)=\kappa(-K_X)$ by \cite[Theorem~10.5]{Iit82}. 
Also, one can easily see that the pullback by a universal homeomorphism 
preserves the Iitaka--Kodaira dimension, so 
$$
\kappa(-K_F)
=\kappa(-K_{F\times_k B})
=\kappa(-\tau^*\sigma^*K_X)
=\kappa(-K_X). 
$$
Since $\hat q(F)=0$, we have $F\cong\mathbb P^1$, so $\kappa(-K_F)=1$. 
However, by construction, one can easily check that $\kappa(-K_X)=0$, 
a contradiction. 
\end{eg}
\begin{eg} \label{eg:torsor need}
This example shows that ``the torsor part'' of Theorems~\ref{thm:F-split semi-ample}, \ref{thm:F-split nef}, \ref{thm:nef tangent} and \ref{thm:hat q=1} cannot be dropped 
(see also \cite[\S 13]{PZ19} and \cite{Lam20}). 
More precisely, there exists an $F$-split projective surface $X$ 
over $\overline{\mathbb F_p}$ such that 
\begin{itemize}
\item the tangent bundle $T_X$ of $X$ is nef, 
\item $-K_X$ is semi-ample, 
\item $\hat q(X)=1$ and 
\item there is no finite \'etale morphism 
$\mathbb P^1\times_k E\to X$ from the product of $\mathbb P^1$ and an elliptic 
curve $E$. 
\end{itemize}

Let $Y$ be an ordinary elliptic curve over $\overline{\mathbb F_p}$. 
Let $\mathcal L$ be a torsion line bundle of order $p$ 
(since $Y$ is ordinary, such a line bundle exists).
Set $X:=\mathbb P_Y(\mathcal O_Y\oplus\mathcal L)$. 
Then $X$ is $F$-split by \cite[Proposition~3.1]{GT16}. 
Also, $-K_X$ is semi-ample, since 
$
X_{Y^1}
\cong \mathbb P_Y(\mathcal O_Y\oplus\mathcal O_Y)
\cong \mathbb P^1 \times_k Y. 
$
The nefness of $T_X$ and $\hat q(X)=1$ follows from the same argument as 
that in Example~\ref{eg:indecomp}. 
Suppose that there exists a finite \'etale morphism 
$\pi:Z:=\mathbb P^1\times_k E\to X$. 
Then one can easily see that there exists a finite \'etale morphism 
$\sigma:E\to Y$ such that $\pi$ coincides with the base change of $\sigma$ 
along $f$. 
Then 
$$
3=h^0\left(\mathbb P^1, \omega_{\mathbb P^1}^{-1}\right)
=h^0\left(Z,\omega_Z^{-1}\right) 
=h^0\left(X_E, \omega_{X_E}^{-1} \right)
=h^0(E, \sigma^*\mathcal L^{-1}\oplus\mathcal O_E\oplus\sigma^*\mathcal L), 
$$
so $\sigma^*\mathcal L\cong\mathcal O_E$. 
On the other hand, by \cite[Corollary~1.7]{Oda71}, 
$\mathcal O_Y$ and $\mathcal L$ are direct summands of ${F_Y}_*\mathcal O_Y$. 
Consider the following commutative diagram:
$$
\xymatrix{ 
	E \ar[r]^{F_E} \ar[d]_\sigma & E \ar[d]^\sigma \\ 
	Y \ar[r]_{F_Y} & Y
}
$$
This is cartesian, since $\sigma$ is \'etale. 
Thus, 
$$
\mathcal O_E^{\oplus 2}
\cong 
\mathcal O_E \oplus \sigma^* \mathcal L
\cong \sigma^*(\mathcal O_Y \oplus \mathcal L)
\subseteq \sigma^* {F_Y}_* \mathcal O_Y
\cong {F_E}_*\mathcal O_E,  
$$
so $h^0(E, \mathcal O_E)\ge 2$, a contradiction. 
\end{eg}

\begin{eg} \label{eg:not SFR}
This example shows that if the geometric generic fiber is not strongly $F$-regular, then there exists an algebraic fiber space that does not split after taking the base change along any proper surjective morphism. 

Let $k$ be an algebraically closed field of characteristic $p=2$ or $3$. 
Then it is known that there exists an algebraic fiber space $f:X\to Y$, 
where $X$ is a Raynaud surface and 
$Y$ is a Tango curve of genus at least two, 
such that 
\begin{itemize}
\item the geometric generic fiber $X_{\overline\eta}$ has a cusp, and 
\item $-K_{X/Y}\sim f^*D$ for an ample divisor $D$ on $Y$ 
\end{itemize}
(See \cite{Tan72} and either \cite{Ray78}, \cite{Muk13} or \cite[Theorem~3.6]{Xie10}).
Thus, $X_{\overline\eta}$ is not strongly $F$-regular (and not $F$-pure) 
and $-K_{X/Y}$ is nef. 
Suppose that there exists a proper surjective morphism $g:Z\to Y$ such that 
there exists an isomorphism $\sigma:X\times_Y Z \to Z\times_k F$ 
of $Z$-schemes, where $F$ is a closed fiber of $f$. 
Note that $K_F\sim 0$, since $K_{X/Y}\sim f^*(-D)$. 
We use the following notation:
$$
\xymatrix{
	X \ar[d]_-f \ar@{}[dr]|\square & X\times_Y Z \ar[l]_-\varphi \ar[d] \ar[r]_-\cong^-\sigma & Z\times_k F \ar@{}[dr]|\square \ar[d]_-h \ar[r]^-i & F \ar[d] \\
	Y  & Z \ar[l]^-g \ar@{=}[r] & Z \ar[r] & \mathrm{Spec}\,k
}
$$
Then 
$$
-h^* g^*D
\sim -\sigma_*\varphi^* f^*D
\sim \sigma_*\varphi^* K_{X/Y} 
\sim \sigma_*K_{X\times_Y Z/Z} 
\sim K_{Z\times_k F/Z}
\sim i^*K_F
\sim 0.
$$
This is a contradiction, as $h^*g^*D$ has positive Iitaka--Kodaira dimension. 
\end{eg}
\begin{eg} \label{eg:F-pure but not SFR}
This example implies that there exists a pair $(X,\Delta)$ and an algebraic fiber space $f:X\to Y$ such that 
\begin{itemize}
\item $-K_{X/Y}-\Delta$ is nef, 
\item $\left(X_{\overline\eta},\Delta|_{X_{\overline\eta}} \right)$ is $F$-pure but not strongly $F$-regular, where $X_{\overline\eta}$ is the geometric generic fiber, and 
\item $f$ does not split after taking the base change along any proper surjective morphism. 
\end{itemize} 

Let $k$ be a perfect field of positive characteristic. 
Put $Y:=\mathbb P^1$ and 
$X:=\mathbb P_Y(\mathcal O_{\mathbb P^1}\oplus\mathcal O_{\mathbb P^1}(-1))$.  
Let $f:X\to Y$ be the natural morphism. 
Let $C_-$ (resp. $C_0$) be the section of $f$ 
corresponding to the quotient 
$\mathcal O_{\mathbb P^1}\oplus \mathcal O_{\mathbb P^1}(-1) 
\twoheadrightarrow \mathcal O_{\mathbb P^1}(-1)$ 
(resp. 
$\mathcal O_{\mathbb P^1}\oplus \mathcal O_{\mathbb P^1}(-1) 
\twoheadrightarrow \mathcal O_{\mathbb P^1}$). 
Then $C_0\sim C_- +f^*y$, where $y$ is a $k$-rational point of $Y$. 
Set $\Delta:=C_-$. 
Then 
$$
-K_{X/Y}-\Delta
\sim 2C_- +f^*y -C_-
=C_- +f^*y
\sim C_0. 
$$
Since $(C_0)^2=(C_-+f^*y)\cdot C_0 =C_-\cdot C_0 +f^*y \cdot C_0 =1$, 
$-K_{X/Y}-\Delta$ is nef (and big). 
However, by construction, $f$ does not split after taking the base change 
along any proper surjective morphism $Z\to Y$. 
\end{eg}
\bibliographystyle{alpha}
\bibliography{ref.bib}
\end{document}